\documentclass{amsart}
\ifdefined\synctex
\fi
\usepackage[utf8]{inputenc}
\usepackage{amsmath, mathtools, amssymb, amsthm}
\usepackage{tikz-cd}
\usepackage{mathrsfs}
\usepackage{hyperref}
\usepackage{tikz}
\usepackage{array, booktabs}
\newtheorem{theorem}{Theorem}[section]
\newtheorem{proposition}[theorem]{Proposition}
\newtheorem{lemma}[theorem]{Lemma}
\newtheorem{corollary}[theorem]{Corollary}

\theoremstyle{definition}
\newtheorem{definition}[theorem]{Definition}
\newtheorem{remark}[theorem]{Remark}

\newtheorem{situation}[theorem]{Situation}
\newtheorem{defexam}[theorem]{Definition and Example}

\newcommand{\Sec}{\mathrm{Sec}}
\newcommand{\dualproj}{(\mathbb{P}^{N-1})^{\vee}}
\newcommand{\percoh}{{}^{\mathfrak{p}}\mathcal{H}}
\newcommand{\IC}{\mathrm{IC}}

\title[Limiting Hodge Structures for Cubic Hypersurfaces of Secant Type]{Limiting Hodge Structures for Cubic Hypersurfaces of Secant Type}

\author{Renjie Lyu}
\address{Mathematical School of Science, Xiamen University, Xiamen, 361005, China}
\email{r.lyu@xmu.edu.cn}

\author{Xuanyu Pan}
\address{School of Mathematical Sciences, University of Chinese Academy of Sciences, Beijing 100049, China}
\email{panxuanyu@ucas.ac.cn}


\begin{document}

\begin{abstract}
The limiting Hodge structure relates the geometric degeneration of 
varieties to its Hodge-theoretic limit. We study degenerations of 
cubic hypersurfaces whose singular central fibers are secant cubics. 
The related degenerations of cubic threefolds and fourfolds have 
been extensively studied for the compactification of the moduli spaces.
In the two cases, the secant cubics are constructed by the rational 
normal quartic and the Veronese surface, respectively. The Veronese 
surface is a Severi variety, and the rational normal quartic is a 
hyperplane section of it. In eariler work, we investigated degenerations 
to the secant cubics of higher-dimensional Severi varieties. Here we 
study the limiting Hodge structures for degenerations to their hyperplane sections.
\end{abstract}

\maketitle
\tableofcontents

\section{Introduction}
\label{sec:introduction}

Variations of Hodge structure provide a transcendental tool for
studying deformations of algebraic varieties. For a smooth projective
family, the cohomology groups of the fibers form a local system. Its
associated holomorphic bundle carries the Hodge filtration and the
Gauss--Manin connection; together these define a variation of Hodge
structure.

For a variation of Hodge structure on the punctured disk \(\Delta^*\), 
Schmid's nilpotent orbit theorem~\cite{Schmid-VHS-73} describes the 
asymptotic behavior of the polarized Hodge structures and gives a 
mixed Hodge structure on the canonical fiber of the local system. 
When the variation arises from a one-parameter degeneration
\(f:\mathcal{X}\rightarrow\Delta\), the resulting limit mixed Hodge
structure is related to the topology and geometry of the singular
central fiber \(\mathcal{X}_0\). If \(\mathcal{X}_0\) is a simple
normal crossings (SNC) divisor in the total space \(\mathcal{X}\),
Steenbrink's geometric construction~\cite{Steenbrink-LMHS} realizes
this structure through the relative de Rham and nearby-cycle complexes.

We study degenerations to secant cubic hypersurfaces associated with
Severi varieties. A Severi variety \(S\subset\mathbb P^{n+2}\) is a
smooth, nondegenerate projective variety of dimension \(2n/3\) whose
secant variety is a proper subset of \(\mathbb P^{n+2}\). Zak
classified the four Severi varieties, up to projective equivalence,
as follows~\cite[Thm. 4.7]{Zak93}:
\[
v_2(\mathbb P^2)\subset \mathbb{P}^5,\qquad \mathbb P^2\times \mathbb P^2\subset \mathbb{P}^8, \qquad
\mathrm{Gr}(2,6)\subset \mathbb{P}^{14}, \qquad E_6/P_1\subset \mathbb{P}^{26}.
\]
For each such \(S\), its secant variety is a cubic hypersurface with
singular locus \(S\)~\cite[\S III,~Thm. 2.9]{Zak93}.

The Veronese surface \(v_2(\mathbb P^2)\subset \mathbb{P}^5\)
plays a role in the compactification of the moduli space of cubic 
\(4\)-folds~\cite{Laza-moduli-cubic-09,Laza-cubic-period-10,Looijenga-period-cubic-09}.
The secant variety of the Veronese surface is a chordal cubic
fourfold, which is semistable in the GIT moduli space of cubic
fourfolds. For a smoothing determined by a smooth cubic fourfold
\(X_\infty\), Hassett and Laza computed the limit mixed Hodge
structure \(\mathrm{H}^4_{\lim}\). It contains
\(\mathrm H^2(Y)_{\mathrm{prim}}(-1)\) as a sub-Hodge structure,
where \(Y\) is the degree-two \(K3\) surface obtained as the double
cover of \(\mathbb P^2\) branched along the sextic curve
\(B=v_2(\mathbb P^2)\cap X_\infty\); see~\cite[Thm. 4.4.1]{Hassett-cubic-00}
and~\cite[Thm. 4.1.1]{Laza-cubic-period-10}. This description
characterizes the limit period point of the degeneration of cubic
fourfolds. Laza used it to prove a period isomorphism between the
GIT moduli space and Looijenga's compactification of the period
domain associated with the arrangement of discriminant-two divisors.

For a higher-dimensional Severi variety \(S\), we studied
degenerations of smooth cubic \((n+1)\)-folds with \(n+1\in \{7,13,25\}\)
whose central fiber is the secant cubic \(\Sec(S)\). In our earlier
paper~\cite{Lyu-Zheng-secant}, we showed that the limit mixed Hodge
structure \(\mathrm{H}^{n+1}_{\lim}\) has the weight filtration
\[
    W_{n}\subset W_{n+1}\subset W_{n+2} = \mathrm{H}^{n+1}_{\lim}.
\]
Let \(B\) be the smooth divisor in \(S\) cut out by the smooth cubic
\((n+1)\)-fold \(X_{\infty}\) that defines the smoothing direction of
\(\Sec(S)\). The graded piece
\(\mathrm{Gr}^W_{n+1} \mathrm{H}^{n+1}_{\lim}\), carrying a canonical 
polarized Hodge structure of weight \(n+1\), is isomorphic to the 
middle cohomology \(\mathrm{H}^{\dim B}(B;\mathbb{Q})\) of \(B\). 
The remaining two graded pieces 
\(\mathrm{Gr}^W_{k}\mathrm{H}^{n+1}_{\lim}\) for \(k=n,n+2\) are 
one-dimensional Tate structures. Unlike the cubic fourfold case, 
\(\mathrm{H}^{n+1}_{\lim}\) is not pure in these higher-dimensional 
cases. Nevertheless, its central weight piece still records geometric 
information about the degeneration.

We later learned that an analogous phenomenon for cubic threefolds
appears in the work of Collino~\cite{Collino-Fano-I} and of Allcock,
Carlson, and Toledo~\cite{ACT11}. The chordal cubic threefold is the
secant variety \(\Sec(C)\) of the rational normal quartic
\(C=v_4(\mathbb P^1)\subset\mathbb P^4\). Fix a general cubic
threefold \(X_\infty\) and consider the smoothing of \(\Sec(C)\)
determined by \(X_\infty\). Collino proved that the associated limit
mixed Hodge structure \(\mathrm H^3_{\lim}\) is isomorphic to
\(\mathrm H^1(\Gamma;\mathbb Q)(-1)\), where \(\Gamma\) is the
genus-five hyperelliptic curve obtained as the double cover of \(C\)
branched at the twelve points of \(C\cap X_\infty\). This result was
used to construct the moduli space of cubic threefolds as a ball
quotient~\cite{ACT11}.

The cubic threefold and fourfold cases share a common feature: the
limiting Hodge structure is governed by the ``infinitesimal base
locus'' determined by the smoothing direction.

The rational normal quartic is a hyperplane section of the Veronese
surface. It is therefore natural to ask whether the same pattern
persists for hyperplane sections of higher-dimensional Severi
varieties. A general hyperplane section of a Severi variety is a
rational homogeneous space. Up to projective equivalence, these
hyperplane sections are as follows:
\begin{itemize}
    \item the rational normal quartic \(v_4(\mathbb{P}^1)\subset \mathbb{P}^4\);
    \item the projectivized cotangent bundle \(\mathbb P(T^*_{\mathbb P^2})\subset \mathbb{P}^7\);
    \item the symplectic Grassmannian of lines \(\mathrm{Gr}_{\omega}(2,6)\subset \mathbb{P}^{13}\);
    \item the homogeneous space \(F_4/P_4\subset \mathbb{P}^{25}\) with the exceptional group \(F_4\).
\end{itemize}
These rational homogeneous spaces and the four Severi varieties
correspond to the two adjacent rows at the top of the Freudenthal--Tits magic
square of division algebras and their Lie algebras
(Table~\ref{tab:magic-square})~\cite{Landsberg-Manivel01}.
\begin{table}[htbp]
\centering
\caption{Freudenthal-Tits Magic Square}
\label{tab:magic-square}
\begin{tabular}{c|cccc}
\toprule
 $\mathfrak{Der}(A/B)$ & $\mathbb{R}$ & $\mathbb{C}$ & $\mathbb{H}$ & $\mathbb{O}$ \\
\midrule
$\mathbb{R}$ & $\mathfrak{so}_3$ & $\mathfrak{sl}_3$ & $\mathfrak{sp}_6$ & $\mathfrak{f}_4$ \\
$\mathbb{C}$ & $\mathfrak{sl}_3$ & $\mathfrak{sl}_3 \oplus \mathfrak{sl}_3$ & $\mathfrak{sl}_6$ & $\mathfrak{e}_6$ \\
$\mathbb{H}$ & $\mathfrak{sp}_6$ & $\mathfrak{sl}_6$ & $\mathfrak{so}_{12}$ & $\mathfrak{e}_7$ \\
$\mathbb{O}$ & $\mathfrak{f}_4$ & $\mathfrak{e}_6$ & $\mathfrak{e}_7$ & $\mathfrak{e}_8$ \\
\bottomrule
\end{tabular}
\end{table}
If \(V\) is such a hyperplane section, then \(\Sec(V)\) is again a
cubic hypersurface (Lemma~\ref{lem:sec-hyp}). Motivated by the
rational normal quartic case and our earlier work, we investigate
degenerations to \(\Sec(V)\) and compare the resulting limit mixed
Hodge structures with those arising from Severi varieties.

Let \(V \subset \mathbb{P}^{n+1}\) be a general hyperplane section of 
a Severi variety \(S\subset \mathbb{P}^{n+2}\). Assume that 
\(d = \dim V \in \{3, 7, 15\}\). The secant variety \(X_0=\Sec(V)\) 
is a cubic hypersurface of dimension \(n\in \{6, 12, 24\}\). 
We consider the pencil of cubic \(n\)-folds
\begin{equation}
\label{eqs:pencil-cubic}
    \mathcal X=\{F+tG=0\}\subset\mathbb {P}^{n+1}\times\Delta
\end{equation}
where \(F\) is the defining polynomial of \(X_0\) and the polynomial 
\(G\) defines a general smooth cubic hypersurface 
\(X_\infty\subset\mathbb P^{n+1}\). Choose \(G\) generally so that
\(\mathcal X_t\) is smooth for all sufficiently small \(t\in\Delta^*\).
After the base change \(t=s^2\) and the blow-up along
\(V\subset X_0\), one obtains a semistable model \(\pi:\mathfrak X\to\Delta\).
The main result is the following.
\begin{theorem}[= Theorem~\ref{thm:main-limit-mhs}]
\label{thm:intro-main}
The limit mixed Hodge structure \(\mathrm{H}^n_{\lim}\) associated 
with the above semistable degeneration is a pure Hodge structure of 
weight \(n\). Set \(B=V\cap X_\infty\). Then there is an isomorphism
of Hodge structures
\[
    \mathrm{H}^n_{\lim}\cong \mathrm{H}^{d-1}(B;\mathbb Q)_{\mathrm{van}}(-m)
    \oplus \mathrm{IH}^n(X_0;\mathbb Q)_v ,
\]
where \(m=\frac{d+5}{4}\in \{2, 3, 5\}\) and 
\[
\begin{split}
    \mathrm{H}^{d-1}(B;\mathbb Q)_{\mathrm{van}} &= \ker\{i_*\colon \mathrm{H}^{d-1}(B;\mathbb Q)\to \mathrm{H}^{d+1}(V;\mathbb Q)(1)\},\\
    \mathrm{IH}^n(X_0;\mathbb Q)_v &= \ker\{\mathrm{IH}^n(X_0;\mathbb Q)\to \mathrm{H}^{d+1}(V;\mathbb Q)(-m+1)\}.
\end{split}
\]
The Hodge structure on \(\mathrm{IH}^n(X_0;\mathbb Q)_v\) is a direct
sum of \(k\) copies of \(\mathbb Q\left(-\frac{n}{2}\right)\),
where \(k=\dim \mathrm{IH}^n(X_0;\mathbb Q)_v\).
\end{theorem}
Compared with the case of Severi varieties~\cite[Thm. 1.1]{Lyu-Zheng-secant},
taking hyperplane sections changes the nilpotency index.
However, the key part of the limiting Hodge structure is still 
governed by the infinitesimal base locus \(B\) of the deformation 
direction. 

This parallel behavior reflects the similar geometry of the semistable
reductions. The central fiber of the semistable model
\(\pi:\mathfrak X\to\Delta\) has two irreducible components \(D_1\)
and \(D_2\). Here \(D_1\) is isomorphic to the blow-up of \(X_0\) 
along \(V\) and \(D_2\) is a quadric fibration over \(V\) whose 
discriminant locus is \(B\). The associated monodromy weight spectral 
sequence computes the limit mixed Hodge structure \(\mathrm{H}^n_{\lim}\). 
The main technical task is to compute the cohomology of \(D_1\),
\(D_2\), and \(D_1\cap D_2\). Projective duality describes the
geometry of \(D_1\). In~\cite{Lyu-Zheng-secant},
we generalized Beauville's method to compute the cohomology of quadric 
fibrations. In this paper, we use the general framework of the
decomposition theorem of perverse sheaves. Our calculation shows that 
the cohomology of the quadric fibration \(D_2\) captures the essential
part of \(\mathrm{H}^n_{\lim}\).

The paper is organized as follows. Section~\ref{sec:hodge-background}
reviews degenerations of Hodge structures and constructions of limit
mixed Hodge structures using Deligne's canonical extension, nearby
cycles, and the logarithmic de Rham complex. Section~\ref{sec:deg-cubic}
studies the geometry of hyperplane sections of Severi varieties and
constructs the semistable model. Section~\ref{sec:quadric-fibrations}
establishes the decomposition for the quadric fibrations. Finally,
Section~\ref{sec:proof-main} proves Theorem~\ref{thm:intro-main}.

\section{Limits of Hodge structures}
\label{sec:hodge-background}

\subsection{Variations of Hodge structure}
We review degenerations of Hodge structures and the limit mixed
Hodge structure. For background, see
\cite{Schmid-VHS-73,Steenbrink-LMHS,Mixed-HS}.

Let \(f\colon X^*\to \Delta^*\) be a smooth projective map from a complex 
manifold \(X^*\) to a punctured disk \(\Delta^*\). Let 
\(A\subset \mathbb{R}\) be a subring (usually \(\mathbb{Z}\) or 
\(\mathbb{Q}\)). The cohomology groups \(\mathrm{H}^k(X_t; A)\), for
\(t\in\Delta^*\), carry canonical Hodge structures and form a local
system as \(t\) varies. Together with the following data, they define
a \emph{variation of Hodge structure of weight \(k\)} on \(\Delta^*\):
\begin{itemize}
    \item the local system \(\mathbb{V}\coloneqq R^kf_*\underline{A}_{X^*}\)
    of cohomology groups of the fibers;
    
    \item the holomorphic vector bundle 
    \(\mathcal{V}\coloneqq \mathbb{V}_{\mathbb C}\otimes_{\mathbb C} \mathcal{O}_{\Delta^*}\),
    which is endowed with the Gauss--Manin connection \(\nabla\) such
    that \(\mathcal{V}^{\nabla} = \mathbb{V}_{\mathbb C}\). The holomorphic bundle 
    \(\mathcal{V}\) is canonically isomorphic to the hypercohomology 
    sheaf \(\mathbb{R}^kf_*\Omega^{\bullet}_{X^*/\Delta^*}\) of the 
    relative de Rham complex \(\Omega^{\bullet}_{X^*/\Delta^*}\);

    \item the decreasing Hodge filtration \(\mathcal{F}^p\subset \mathcal{V}\) 
    for \(0\leq p\leq k\) defined by
    \[
    \mathcal{F}^p\coloneqq \mathbb{R}^kf_*\Omega^{\bullet \geq p}_{X^*/\Delta^*}
    \]
    where \(\Omega^{\bullet \geq p}_{X^*/\Delta^*}\) is the stupid truncation of \(\Omega^{\bullet}_{X^*/\Delta^*}\).
    This filtration satisfies Griffiths transversality:
    \[
        \nabla(\mathcal{F}^p)\subset \Omega^1_{\Delta^*}\otimes \mathcal{F}^{p-1}.
    \] 
\end{itemize}
If we replace \(\mathbb{V}\) by the locally constant sheaf of primitive 
cohomology groups of \(X_t\), the polarization on 
\(\mathrm{H}^k(X_t; A)_{\mathrm{prim}}\) induces 
\begin{itemize}
    \item the \((-1)^k\)-symmetric bilinear form of local systems
    \[
        Q\colon \mathbb{V}\otimes \mathbb{V}\to \underline{A}(-k).
    \]
\end{itemize}
These data \((\mathbb{V}, \mathcal{F}^{p}, \nabla, Q)\) form
an \(A\)-polarized variation of Hodge structure of weight \(k\) on
\(\Delta^*\).

Deligne's canonical extension describes the asymptotic behavior of
the Hodge structures as the parameter approaches zero and the
corresponding limit mixed Hodge structure. An extension of the
integrable connection \((\mathcal{V}, \nabla)\) consists of a holomorphic
bundle \(\widetilde{\mathcal{V}}\) on \(\Delta\), equipped with a
logarithmic connection
\[
\widetilde{\nabla}\colon \widetilde{\mathcal{V}}\to \Omega^1_{\Delta}(\log 0)\otimes\widetilde{\mathcal{V}}
\] 
such that 
\((\widetilde{\mathcal{V}}, \widetilde{\nabla})|_{\Delta^*} \cong (\mathcal{V}, \nabla)\). 
For any \(t\in \Delta^*\), the monodromy transformation 
\[
T_t\colon \mathrm{H}^k(X_t; A) \to \mathrm{H}^k(X_t; A)
\]
associated with the local system \(\mathbb{V}\) plays a central role
in the construction of the limit mixed Hodge structure.

Let \(\mathfrak{h} = \{u\in \mathbb{C}~|~ \mathrm{Im}(u) > 0\}\) be
the upper half-plane, which is the universal cover of \(\Delta^*\)
via the exponential map
\[ 
e\colon\mathfrak{h}\to \Delta^*, ~e(u)=e^{2\pi i u},
\]
and let \(\tilde{X}^* = X^*\times_{\Delta^*} \mathfrak{h}\). Since
\(\mathfrak{h}\) is contractible, the local system \(e^*\mathbb{V}\)
admits a global trivialization 
\[
   e^*\mathbb{V}\simeq \mathfrak{h}\times \mathbb{V}_{\infty}
\]
whose fiber \(\mathbb{V}_{\infty}\) is isomorphic to the \(A\)-module 
\(\mathrm{H}^k(\tilde{X}^*; A)\). We call \(\mathbb{V}_{\infty}\) the
\emph{canonical fiber} of the local system \(\mathbb{V}\). Fix 
\(t\in \Delta^*\), and choose \(u\in \mathfrak{h}\) such that \(t=e^{2\pi i u}\). 
The natural embedding \(j_u\colon X_t\hookrightarrow \tilde{X}^*, ~x\mapsto (x, u)\) induces an isomorphism
\[
   j^*_u\colon \mathrm{H}^k(\tilde{X}^*; A)\xrightarrow{\sim} \mathrm{H}^k(X_t; A).
\]
Via this identification, the monodromy action on \(\mathrm{H}^k(X_t; A)\)
induces the monodromy transformation  
\[
T\colon \mathrm{H}^k(\tilde{X}^*; A)\to \mathrm{H}^k(\tilde{X}^*; A).
\]
It is independent of the choices of \(t\) and \(u\). 
Consider the canonical diffeomorphism 
\[
h\colon \tilde{X}^*\to \tilde{X}^*, (x, u)\to (x, u+1)
\]
and the induced linear transformation \(h^*\) on \(\mathrm{H}^k(\tilde{X}^*; A)\).
Under this identification one has \(T=h^*\); see~\cite[\S~II, 2.4]{Kulikov-MHS}.



Borel's monodromy theorem~\cite[Thm. 6.1]{Schmid-VHS-73} asserts that
\(T\) is quasi-unipotent and its unipotent part has nilpotency index at
most \(k+1\): the eigenvalues of the semisimple part \(T_s\) have
finite order, and \((T_u-I)^{k+1}=0\). Because \(T_u-I\) is nilpotent,
its logarithm is a finite sum, and we may define the nilpotent map
on the canonical fiber \(\mathbb{V}_{\infty, \mathbb{C}}=\mathrm{H}^k(\tilde{X}^*; \mathbb{C})\):
\[
N\coloneqq -\frac{\log T_u}{2\pi i} =\frac{1}{2\pi i}\sum_{j=1}^{k} \frac{(I-T_u)^{j}}{j}.
\]

Assume that \(T\) is unipotent, i.e., \(T=T_u\). A cohomology class
\(v\in \mathrm{H}^k(\tilde{X}^*; \mathbb{C})\) can be viewed as a 
horizontal section in \(\Gamma(\mathfrak{h}, e^*\mathcal{V})^{\nabla}\),
or a multivalued horizontal section of \(\mathcal{V}\). 
Define a holomorphic section
\[
    \tilde{v}(u)=\mathrm{exp}(2\pi i uN)v(u)
\]
of the holomorphic bundle \(e^*\mathcal{V}\). A horizontal section
satisfies \(Tv(u)=v(u+1)\). It follows that \(\tilde{v}\) is invariant 
under the transformation \(u\mapsto u+1\):
\[
   \tilde{v}(u+1)= \mathrm{exp}(2\pi i (u+1)N)Tv(u)=\tilde{v}(u).
\]
Here we use \(T=\mathrm{exp}(-2\pi iN)\). Hence \(\tilde{v}\) descends 
to a single-valued section of \(\mathcal{V}\) over \(\Delta^*\). 
By \(u=\frac{\log t}{2\pi i}\), let us write 
\(\tilde{v}(t) = t^N v(\frac{\log t}{2\pi i})\). Since \(v\) is horizontal, 
we have
\[
\nabla\tilde{v}(t)=N t^N v(\frac{\log t}{2\pi i})\otimes \frac{dt}{t}=N\tilde{v}(t)\otimes \frac{dt}{t}.
\]
The single-valued sections \(\tilde v\) define a new frame
\(\widetilde{\mathbb{V}}_{\infty,\mathbb{C}}\), which gives a
holomorphic bundle over \(\Delta\):
\[
\widetilde{\mathcal{V}}\coloneqq \widetilde{\mathbb{V}}_{\infty,\mathbb{C}}\otimes_{\mathbb{C}} \mathcal{O}_{\Delta}.
\]
The connection \(\nabla\) naturally extends to a meromorphic 
connection \(\widetilde{\nabla}\) with at most a simple pole at \(t=0\),
and the residue map at \(0\) 
\[
\mathrm{Res}_0(\widetilde{\nabla})\colon \widetilde{\mathcal{V}}(0)\to \widetilde{\mathcal{V}}(0)
\] 
equals the nilpotent monodromy \(N\) via the isomorphism
\(\mathbb{V}_{\infty, \mathbb{C}}\cong \widetilde{\mathcal{V}}(0)\). 
This logarithmic extension \((\widetilde{\mathcal{V}}, \widetilde{\nabla})\) 
is Deligne's canonical extension of \((\mathcal{V}, \nabla)\)
whose residue at \(0\) has zero eigenvalues.

There are two significant filtrations related to the nilpotent monodromy \(N\):
\begin{itemize}
    \item the \emph{limiting Hodge filtration} \(\widetilde{\mathcal{F}}^p\subset \widetilde{\mathcal{V}}\) 
    is the decreasing filtration 
    \[
    \widetilde{\mathcal{F}}^p\coloneqq \lim_{\mathrm{Im}(u)\to \infty} 
    \mathrm{exp}(2\pi i uN)\cdot \mathcal{F}^p_u
    \]
    extending the Hodge filtration \(\mathcal{F}^p\).

    \item Let \(\mathbb{V}_{\infty, \mathbb{Q}}=\mathrm{H}^k(\tilde{X}^*; \mathbb{Q})\).
    The \emph{weight filtration} \(W(N)\) on \(\mathbb{V}_{\infty, \mathbb{Q}}\)
    induced by the nilpotent map \(N\) of nilpotency index at most \(k+1\) is the unique 
    increasing filtration
    \[
    0\subset W_0\subset \cdots \subset W_{2k} = \mathbb{V}_{\infty, \mathbb{Q}}
    \]
    such that \(N(W_l)\subset W_{l-2}\), and the linear map
    \[
    N^l\colon \mathrm{Gr}^W_{k+l} \mathbb{V}_{\infty, \mathbb{Q}}\to \mathrm{Gr}^W_{k-l} \mathbb{V}_{\infty, \mathbb{Q}}(-l)
    \] 
    is an isomorphism for all \(l\geq 0\).
\end{itemize}
Let \(F^p_{\infty}\subset \mathbb{V}_{\infty,\mathbb{C}}\) denote
the limiting Hodge filtration that corresponds to the filtration
\(\widetilde{\mathcal{F}}^p(0)\subset \widetilde{\mathcal{V}}(0)\).
Schmid showed~\cite{Schmid-VHS-73} that the data
\begin{equation}
\label{eqs:mixed-Hodge-data}
    \mathbb{V}^k_{\lim}\coloneqq (\mathbb{V}_{\infty,\mathbb{Q}}, W(N), F^p_{\infty})
\end{equation}
yields a mixed Hodge structure on the canonical fiber, which is the 
limit mixed Hodge structure for the variation
\((\mathbb{V}, \mathcal{F}^{p}, \nabla, Q)\) of weight \(k\) on the 
punctured disk.

\subsection{Logarithmic de Rham complex}
A proper holomorphic map \(f\colon X\to \Delta\) is a semistable 
degeneration if the central fiber \(D = f^{-1}(0)\) is a simple normal 
crossings divisor in \(X\). Locally, this means that one can choose coordinates
of \(X\) around any \(p\in D\) such that \(f(z_1,\ldots, z_n)=z_1z_2\cdots z_v\) 
for some \(v\leq n\). In this geometric setting, Deligne's canonical
extension and the limit mixed Hodge structure can be described in terms 
of the logarithmic de Rham complex.

The relative de Rham complex on \(X\) with logarithmic poles along 
\(D\) is the complex  
\[
    \Omega^{\bullet}_{X/\Delta}(\log D)\coloneqq
    \frac{\Omega^{\bullet}_{X}(\log D)}
    {f^*\Omega^{1}_{\Delta}(\log 0)\wedge
     \Omega^{\bullet-1}_X(\log D)}.
\]
The \(k\)-th relative hypercohomology sheaf 
\(\mathbb{R}^kf_*\Omega^{\bullet}_{X/\Delta}(\log D)\) is a locally 
free \(\mathcal{O}_\Delta\)-module such that
\[
    \mathbb{R}^kf_*\Omega^{\bullet}_{X/\Delta}(\log D)|_{\Delta^*}=\mathbb{R}^kf_*\Omega^{\bullet}_{X^*/\Delta^*}.
\] 
The connecting homomorphism 
\[
    \widetilde{\nabla}\colon \mathbb{R}^kf_*\Omega^{\bullet}_{X/\Delta}(\log D)\to 
    \Omega^1_{\Delta}(\log 0)\otimes \mathbb{R}^kf_*\Omega^{\bullet}_{X/\Delta}(\log D)
\]
associated with the exact sequence of complexes
\[
    0\to f^*\Omega^1_{\Delta}(\log 0)\otimes \Omega^{\bullet}_{X/\Delta}(\log D)[-1]
    \to \Omega^{\bullet}_{X}(\log D)\to \Omega^{\bullet}_{X/\Delta}(\log D)\to 0
\]
defines a meromorphic connection \(\widetilde{\nabla}\) that extends 
the Gauss--Manin connection \(\nabla\) on the de Rham complex 
\(\mathbb{R}^kf_*\Omega^{\bullet}_{X^*/\Delta^*}\). Since 
\(\widetilde{\nabla}\) has a logarithmic pole at the origin, the monodromy
action \(T_t\) for \(t\in \Delta^*\) extends to a limit monodromy 
\(T_0\) satisfying the relation (see~\cite[\S~I,~7.8]{Kulikov-MHS})
\[
    T_0= \mathrm{exp}(-2\pi i \mathrm{Res}_0(\widetilde{\nabla})).
\]
Therefore the logarithmic connection 
\((\mathbb{R}^kf_*\Omega^{\bullet}_{X/\Delta}(\log D), \widetilde{\nabla})\) 
is Deligne's canonical extension of \((\mathcal{V}, \nabla)\). 
The hypercohomology 
\[
\mathbb{H}^k(D, \Omega^{\bullet}_{X/\Delta}(\log D)\otimes_{\mathcal{O}_X} \mathcal{O}_D)
\]
is isomorphic to the canonical fiber 
\(\mathbb{V}_{\infty,\mathbb{C}} = \mathrm{H}^k(\tilde{X}^*; \mathbb{C})\) due to
the nearby cycle complex and Theorem~\ref{thm:nearby-cycl-log-deRham}.
Consider the diagram
\[
\begin{tikzcd}[column sep=large, row sep=large]
\widetilde{X}^{*}
  \arrow[r]
  \arrow[rr, bend left=20, "\widetilde{j}"]
  \arrow[d]
&
X^{*} \arrow[r, hook] \arrow[d]
&
X \arrow[d]
&
D \arrow[l, hook', "i"'] \arrow[d]
\\
\mathfrak h \arrow[r, "e"]
&
\Delta^{*} \arrow[r, hook]
&
\Delta
&
\{0\} \arrow[l, hook']
\end{tikzcd}
\]
The nearby-cycle complex \(\psi_f \mathbb{C}_{X}\) of the holomorphic
map \(f\colon X\to \mathbb{C}\) is defined to be the constructible 
complex \(i^{-1}R\widetilde j_*\widetilde j^*\underline{\mathbb{C}}_{X}\) 
in the derived category \(D^b_c(D; \mathbb{C})\).
\begin{theorem}\cite[Theorem 11.16]{Mixed-HS}
    \label{thm:nearby-cycl-log-deRham}
There is a quasi-isomorphism
\begin{equation}
\label{eqs:nearby-cycle-qis}
    \psi_f \mathbb{C}_{X}\simeq \Omega^{\bullet}_{X/\Delta}(\log D)\otimes_{\mathcal{O}_X} \mathcal{O}_D
\end{equation}  
in the derived category \(D^b_c(D;\mathbb{C})\). Consequently,
\[
    \mathbb{H}^k(D, \Omega^{\bullet}_{X/\Delta}(\log D)\otimes_{\mathcal{O}_X} \mathcal{O}_D)
    \cong \mathbb{H}^k(D, i^{-1}R\widetilde{j}_* \widetilde{j}^*\underline{\mathbb{C}}_{X})
    \cong \mathrm{H}^k(\tilde{X}^*; \mathbb{C}).
\]
\end{theorem}
The Hodge filtration \(F^{\bullet}\) on \(\Omega^{\bullet}_{X/\Delta}(\log D)\otimes \mathcal{O}_D\) 
is given by the stupid truncation
\[
  \Omega^{\bullet\geq p}_{X/\Delta}(\log D)\coloneqq \cdots \to 0\to \Omega^{p}_{X/\Delta}(\log D)\to \Omega^{p+1}_{X/\Delta}(\log D)\to \cdots.
\]
To describe the weight filtration on 
\(\Omega^{\bullet}_{X/\Delta}(\log D)\otimes \mathcal{O}_D\), we
consider the double complex \((A^{\bullet,\bullet}, d', d'')\):
\[
\begin{split}
            & A^{p,q} =\Omega_X^{p+q+1}(\log D)/W_q\Omega_X^{p+q+1}(\log D),\\
  &d' \colon  A^{p,q}\to A^{p+1, q}, ~d'(w)  = dw,\\
  &d'' \colon A^{p,q}\to A^{p, q+1}, ~d''(w) = f^*(dt/t)\wedge w,\\
  &d'd''+d''d'=0
\end{split}
\]
Let \(A_{\mathbb{C}}^{\bullet}\) denote the associated single complex 
of the double complex \((A^{\bullet,\bullet}, d', d'')\). The map
\[
\theta\colon \Omega^{p}_{X/\Delta}(\log D)\otimes \mathcal{O}_D\to A^{p,0}, w\to (-1)^pf^*(dt/t)\wedge w
\]
extends to a quasi-isomorphism of complexes
\begin{equation}
\label{eqs:quasiso-rel-deRham}
   \widetilde{\theta}\colon \Omega^{\bullet}_{X/\Delta}(\log D)\otimes \mathcal{O}_D 
   \xrightarrow{\sim} A_{\mathbb{C}}^{\bullet}.
\end{equation}
The complex \(A_{\mathbb{C}}^{\bullet}\) carries the filtration:
\[
\begin{split}
    W_{r} A_{\mathbb{C}}^{\bullet} &= \bigoplus_{p,q\geq 0} W_r A^{p,q},\\
    W_rA^{p,q} &= \operatorname{Im}(W_{r+2q+1}\Omega_X^{p+q+1}(\log D)\to A^{p,q}).
\end{split}
\]
Write \(D=\bigcup D_i\), where the \(D_i\) are its smooth, reduced components.
Use the notation
\[
\begin{split}
    & D_{I}  \coloneqq D_{i_1}\cap \cdots \cap D_{i_m}, ~I=\{i_1, \ldots, i_m\},\quad D(m)   \coloneqq\coprod_{|I|=m} D_I,\\
    & a_{I}  \colon D_{I}\hookrightarrow D,\quad  a_m\coloneqq \coprod_{|I|=m}a_{I}\colon D(m)\to D
\end{split}
\]
The residue map identifies the graded pieces of the weight filtration
\(W_{\bullet}\) on \(A_{\mathbb{C}}^{\bullet}\) with
\begin{equation}
\label{eqs:weight-residue-map}
\mathrm{Gr}^W_r A_{\mathbb{C}}^{\bullet}\simeq \bigoplus_{p\geq \max\{0, -r\}} {a_{2p+r+1}}_*\Omega^{\bullet}_{D(2p+r+1)}[-r-2p] 
\end{equation}

By the quasi-isomorphism~\eqref{eqs:nearby-cycle-qis}, the nearby 
cycle complex \(\psi_f\mathbb{Q}_X = i^{-1}Rk_* k^*\underline{\mathbb{Q}}_{X}\) 
is the underlying \(\mathbb{Q}\)-structure of the relative 
logarithmic de Rham complex
\(\Omega^{\bullet}_{X/\Delta}(\log D)\otimes \mathcal{O}_D\).
Moreover, \(\psi_f\mathbb{Q}_X\) admits a weight filtration that is 
compatible with the weight filtration \(W_{\bullet}\) on 
\(\Omega^{\bullet}_{X/\Delta}(\log D)\otimes \mathcal{O}_D\).

The complex \(A^{\bullet}_{\mathbb{Q}}\in D^b_c(D;\mathbb{Q})\)
associated with \(i^{-1}Rj_*\mathbb{Q}_{X^*}\) admits a natural
quasi-isomorphism
\begin{equation}
\label{eqs:quasi-iso-Q-nearby-cyc}
 A^{\bullet}_{\mathbb{Q}}\xrightarrow{\sim} \psi_f\mathbb{Q}_X.
\end{equation}
The canonical filtration on the complex \(i^{-1}Rj_*\mathbb{Q}_{X^*}\)
induces an increasing filtration \(W_{\bullet,\mathbb{Q}}\) on
\(A^{\bullet}_{\mathbb{Q}}\) such that
\begin{equation}
\label{eqs:Q-wt-fil}
    \mathrm{Gr}_r^{W} A^{\bullet}_{\mathbb{Q}}\cong 
    \bigoplus_{p\geq \max\{0, -r\}} {a_{2p+r+1}}_*\mathbb{Q}_{D(2p+r+1)}(-r-p)[-r-2p].
\end{equation}
Moreover, the induced quasi-isomorphism
\begin{equation}
\label{eqs:complex-Q-structure}
A^{\bullet}_{\mathbb{Q}}\otimes_{\mathbb{Q}} \mathbb{C}\simeq \psi_f\mathbb{C}_X
\simeq \Omega^{\bullet}_{X/\Delta}(\log D)\otimes \mathcal{O}_D
\end{equation}
is a filtered quasi-isomorphism with respect to the filtrations
\(W_{\bullet,\mathbb{Q}}\otimes_{\mathbb{Q}} \mathbb{C}\) and \(W_{\bullet}\).

\begin{theorem}\cite[Corollary 11.23]{Mixed-HS}
\label{thm:mono-wt-fil}
The filtrations \(W_{\bullet, \mathbb{Q}}\) on \(A^{\bullet}_{\mathbb{Q}}\)
and \(F^{\bullet}\) on \(\Omega^{\bullet}_{X/\Delta}(\log D)\otimes \mathcal{O}_D\)
give the following two spectral sequences:
\begin{enumerate}
    \item the \emph{monodromy weight spectral sequence}
\[
   _{W} E^{-r, m+r}_1= \bigoplus_{p\geq \max\{0, -r\}} \mathrm{H}^{m-r-2p}(D(2p+r+1); \mathbb{Q})(-r-p)\Rightarrow 
   \mathrm{H}^{m}(\tilde{X}^*; \mathbb{Q})
\]
that degenerates at \(E_2\);
    \item the \emph{Hodge spectral sequence}
    \[
        _{F} E^{p,q}_1 = \mathrm{H}^q(D, \Omega^{p}_{X/\Delta}(\log D)\otimes \mathcal{O}_D)\Rightarrow \mathrm{H}^{p+q}(\tilde{X}^*; \mathbb{C})
    \]
    that degenerates at \(E_1\).
\end{enumerate}
\end{theorem}

The data
\[
    (\psi_f\mathbb{Q}_X, (A^{\bullet}_{\mathbb{Q}}, W_{\bullet,\mathbb{Q}}), (\Omega^{\bullet}_{X/\Delta}(\log D)\otimes \mathcal{O}_D, F^{\bullet}, W_{\bullet}))
\]
together with the quasi-isomorphisms~\eqref{eqs:nearby-cycle-qis},
\eqref{eqs:quasiso-rel-deRham} and~\eqref{eqs:quasi-iso-Q-nearby-cyc} 
form a mixed Hodge complex on \(D\). Moreover, the induced mixed Hodge 
structure on the canonical fiber 
\(\mathbb{V}_{\infty,\mathbb{Q}}=\mathrm{H}^k(\tilde{X}^*; \mathbb{Q})\) 
coincides with Schmid's limit mixed Hodge structure~\eqref{eqs:mixed-Hodge-data}:
\[
\mathbb{V}^k_{\lim} = (\mathbb{V}_{\infty,\mathbb{Q}}, W(N), F^p_{\infty}),
\]
as shown in~\cite[Theorem 5.9]{Steenbrink-LMHS}.

\section{Cubic hypersurfaces of secant type}
\label{sec:deg-cubic}
\subsection{Severi varieties and their secant varieties}
\label{subsec:severi-hyperplane}
We briefly review Severi varieties, their hyperplane sections,
and the geometry of the associated secant varieties.

Let \(k\) be a field of characteristic zero. Let \(X\subset \mathbb{P}^N_k\)
be an irreducible, nondegenerate and nonsingular closed subvariety. The 
\emph{secant variety} \(\Sec(X)\) is the Zariski closure of the union
of all secant lines (chords) to \(X\) in \(\mathbb{P}^N_k\):
\[
\Sec(X)\coloneqq \overline{\bigcup_{x\neq y\in X} \langle x, y\rangle}
\]
(a limit of secant lines may be a tangent line to \(X\)). The secant
variety is proper if and only if there exists a linear
projection \(\pi\colon \mathbb{P}^N\dashrightarrow \mathbb{P}^{N-1}\) 
such that \(X\) maps isomorphically onto its image \(\pi(X)\). Zak proved 
that \(\Sec(X)= \mathbb{P}^N\) if \(\dim X > \frac{2(N-2)}{3}\) 
(cf. \cite[Thm.~7.6]{Fulton-Lazasfeld}), which leads to the notion 
of a \emph{Severi variety}. We say \(X\subset \mathbb{P}^N\) is a 
Severi variety if \(\dim X = \frac{2(N-2)}{3}\) and \(\Sec(X)\neq \mathbb{P}^N\). 
Zak~\cite[\S III,~Thm. 2.9]{Zak93} classified
the Severi varieties, up to projective equivalence, as follows:
\begin{enumerate}
    \item the Veronese surface \(v_2(\mathbb{P}^2) \subset \mathbb{P}^5\); 
    \item the Segre variety \(\mathbb{P}^2 \times \mathbb{P}^2 \subset \mathbb{P}^8\); 
    \item the Grassmannian of lines \(\mathrm{Gr}(2,6) \subset \mathbb{P}^{14}\); 
    \item the Cartan variety \(E_6/P_1 \subset \mathbb{P}^{26}\) of the exceptional group \(E_6\). 
\end{enumerate}
Each Severi variety \(X\subset \mathbb{P}^N\) is the projectivized
rank-one locus in the complexified Jordan algebra
\(J_3(\mathbb{A})\otimes_{\mathbb R}\mathbb C\), where
\(J_3(\mathbb{A})\) is the Jordan algebra of \(3\times3\) Hermitian
matrices over
\(\mathbb{A}\in \{\mathbb{R},\mathbb{C},\mathbb{H},\mathbb{O}\}\).
In particular, \(2\dim_{\mathbb{R}}\mathbb{A}=\dim_{\mathbb C}X\).

In this paper, we focus on general hyperplane sections of Severi 
varieties. From the viewpoint of composition algebras, the Severi varieties 
and their hyperplane sections are related to Lie algebras in
the Freudenthal-Tits magic square~\cite{Landsberg-Manivel01}. The hyperplane sections are 
projectively equivalent to the following special varieties:
\begin{enumerate}
    \item the rational normal curve \(v_4(\mathbb{P}^1)\subset \mathbb{P}^4\); 
    \item the projectivized cotangent bundle \(\mathbb{P}(T^*_{\mathbb{P}^2})\subset \mathbb{P}^7\); 
    \item the symplectic Grassmannian of lines \(\mathrm{Gr}_\omega(2,6)\subset \mathbb{P}^{13}\),
    where \(\omega\) is the symplectic \(2\)-form that defines the hyperplane 
    \(\mathbb{P}^{13}\subset \mathbb{P}^{14}\cong \mathbb{P}(\wedge^2 \mathbb{C}^{6})\);
    \item the variety \(F_4/P_4\subset \mathbb{P}^{25}\) of the exceptional group \(F_4\). 
\end{enumerate}
Like the Severi varieties, these hyperplane sections are rational
homogeneous spaces for the following groups, respectively
(see Table~\ref{tab:magic-square}):
\[
    \mathrm{SO}(3), \quad \mathrm{SL}(3),\quad \mathrm{Sp}(6), \quad F_4.
\]

Although these varieties have different geometric descriptions,
their secant varieties share several features. We establish these
features and describe the relevant projective duality.

\begin{lemma}
\label{lem:sec-hyp}
    Let \(S\subset \mathbb{P}^{N}\) be a Severi variety. Let
    \(H\subset\mathbb{P}^N\) be a general hyperplane and set \(V=S\cap H\).
    Let \(\Sec(V)\) be the secant variety of
    \(V\) in \(H\cong \mathbb{P}^{N-1}\). Then
    \(\Sec(V)=\Sec(S)\cap H\). Consequently, \(\Sec(V)\) is a
    cubic hypersurface whose singular locus is \(V\).
\end{lemma}
\begin{proof}
    The inclusion \(\Sec(V)\subset \Sec(S)\cap H\) is immediate.
    For the reverse inclusion, let \(p\in \Sec(S)\cap H\setminus S\).
    For a Severi variety \(S\), the secant locus 
    \[
    Q_p\coloneqq \{x\in S~|~\langle x,p\rangle ~\text{is a secant or tangent line to}~ S\}
    \]
    is a smooth quadric with \(\dim Q_p = \frac{\dim S}{2} > 0\);
    see~\cite[Theorem 2.1]{Lazarsfeld-Van}.
    Choose \(x\in Q_p\cap H\). The line \(\langle x,p\rangle\)
    meets \(Q_p\) in a second point \(y\). Because \(p\) and \(x\)
    lie in \(H\), so does \(y\), and \(y\in S\cap H=V\). Hence \(\langle x,p\rangle\)
    is a chord of \(S\cap H = V\), and \(p\in \Sec(V)\).

    The hypersurface \(\Sec(S)\) is cubic and has singular locus \(S\).
    A general hyperplane section is smooth away from \(S\cap H\), while
    transversality along \(S\) shows that its singular locus is exactly 
    \(V=S\cap H\). Therefore \(\Sec(V)\) is a cubic hypersurface whose
    singular locus is \(V\).
\end{proof}


\subsection{Projective dual variety}
Let \(X\subset \mathbb{P}^N\) be an irreducible and nondegenerate 
closed subvariety. Consider the incidence subvariety
\[
I^{\circ}_X\coloneqq \{(x, [H])\in X\times (\mathbb P^{N})^{\vee} ~|~ \hat{T}_x X\subset H, ~x \in X_{\mathrm{sm}}\}
\] 
Its points correspond to pairs \((x,[H])\) such that \(H\) is tangent
to \(X\) at a smooth point \(x\). The Zariski closure \(I_X\) of
\(I^{\circ}_X\) is called the \emph{conormal variety} of \(X\). Let 
\(\mathrm{pr}_2\colon I_X\to (\mathbb P^{N})^{\vee}\) be the second 
projection. The dual variety \(X^{\perp}\) of \(X\) is defined as 
the closure of the image of \(\mathrm{pr}_2\). In particular, if \(X\)
is a hypersurface, the conormal variety \(I_X\) is the closure of the
graph of the Gauss map on \(X_{\mathrm{sm}}\):
\begin{equation}
\label{eqs:Gauss-map}
    \gamma_{N-1}\colon X\dashrightarrow (\mathbb{P}^N)^{\vee}, \quad \gamma_{N-1}(x)=\hat{T}_x X, ~\forall x\in X_{\mathrm{sm}}.
\end{equation}

It is known that the dual variety of the secant variety 
\(\Sec(S)\subset \mathbb{P}^{N}\) is isomorphic to the Severi variety 
\(S\subset (\mathbb{P}^{N})^{\vee}\) where the embedding is
given by the contragredient representation; see~\cite[Chap.~III,~Thm 2.9]{Zak93}.
The dual variety of \(\Sec(V)\subset \mathbb{P}^{N-1}\) is 
no longer isomorphic to \(V\), see~\cite[Chap.~III,~2.5]{Zak93}. 
We instead describe its dual variety using a general result by 
Tevelev~\cite[Thm. 1.23]{Tevelev:proj-dual}.

Let \(U\) be a linear space and \(W\subset U\) be a linear subspace.
Consider the linear projection 
\(\pi_{\Lambda}\colon \mathbb{P}(U)\dashrightarrow \mathbb{P}(U/W)\)
along the subspace \(\mathbb{P}(W)\). The dual variety
\(\mathbb{P}(W)^{\perp}\subset \mathbb{P}(U)^{\vee}\), consisting of the 
linear functions vanishing on \(W\), is naturally identified with the 
dual space \(\mathbb{P}(U/W)^{\vee}\) as a subspace in \(\mathbb{P}(U)^{\vee}\).
To avoid confusion, we emphasize that the dual space
\(\mathbb{P}(W)^{\vee}\) is different from the dual variety \(\mathbb{P}(W)^{\perp}\). 

\begin{theorem}[Theorem 1.23 in~\cite{Tevelev:proj-dual}]
\label{thm:dual-hyper-sec}
Let \(X\subset \mathbb{P}(U)\) be a projective subvariety disjoint 
from a linear subspace \(\Lambda\coloneqq \mathbb{P}(W)\). Let 
\(\pi_{\Lambda}(X)\subset \mathbb{P}(U/W)\) denote the image of \(X\). 
Assume that \(\pi_{\Lambda}\colon X\to \pi_{\Lambda}(X)\) is an 
isomorphism. Then we have 
\[
    \pi_{\Lambda}(X)^{\perp} = X^{\perp} \cap \mathbb{P}(U/W)^{\vee}.
\]
Consequently, \(\pi_{\Lambda}(X)\subset \mathbb{P}(U/W)\)
is isomorphic to the dual variety of the linear section 
\(X^{\perp} \cap \mathbb{P}(U/W)^{\vee}\).
\end{theorem}

\begin{remark}
    Liu and Zhang~\cite{Liu-Zhang} generalized this relation to the
    discriminant locus of a generic projection of a projective variety
    and its dual variety.
\end{remark}

\begin{corollary}
\label{cor:dual-hypsec-severi}
    Let \(S\subset \mathbb{P}^{N}\) be a Severi variety, and let \(V\) 
    be the section \(S\cap H\) for a general hyperplane \(\mathbb{P}^{N-1}\cong H\subset \mathbb{P}^{N}\). 
    Then the dual variety \(\Sec(V)^{\perp}\subset \dualproj\) is isomorphic 
    to the image of the Severi variety \(S\subset (\mathbb{P}^{N})^{\vee}\) 
    via the projection \(\pi_H\colon (\mathbb{P}^{N})^{\vee}\dashrightarrow \dualproj\)
    where \(H\) denotes a point in \((\mathbb{P}^{N})^{\vee}\).
\end{corollary}
\begin{proof}
    Replace \(\mathbb{P}(U)\) (resp. \(X\)) in Theorem~\ref{thm:dual-hyper-sec}
    by the dual space \((\mathbb{P}^{N})^{\vee}\) (resp. the Severi variety \(S\)). 
    We can choose \(\Lambda = [H]\in (\mathbb{P}^{N})^{\vee}\) to be 
    a general point such that \([H]\) lies outside the secant variety
    of \(S\subset (\mathbb{P}^{N})^{\vee}\). Then \(H\) is identified 
    with \(\mathbb{P}(U/W)^{\vee}\), and 
    \[
        \pi_{[H]}\colon S\to \pi_{[H]}(S)\subset \dualproj
    \]
    is an isomorphism. By Theorem~\ref{thm:dual-hyper-sec},
    \[
        \pi_{[H]}(S)^{\perp}=S^{\perp} \cap H.
    \]
    Since \(S^{\perp}\cong \Sec(S)\subset \mathbb{P}^N\),
    the right-hand side is \(\Sec(V)\) by Lemma~\ref{lem:sec-hyp}. Using 
    the reflexive theorem, we obtain \(\Sec(V)^{\perp} = \pi_{[H]}(S)\subset \dualproj\).
\end{proof}

\begin{lemma}
\label{lem:sm-quad-bdl}
    Let \(S\subset \mathbb{P}^{N}\) be a Severi variety of dimension
    \(d+1\). Let \(H\subset \mathbb{P}^{N}\) be a general hyperplane,
    and set \(V=S\cap H\) and \(Y=\Sec(V)\). The blow-up
    \(\tilde{Y}\) of \(Y\) along the singular locus \(V\) is 
    nonsingular. The exceptional divisor \(E_0\) is a smooth quadric 
    bundle over \(V\) of relative dimension \(\frac{d+1}{2}\).
\end{lemma}
\begin{proof}
Let \(f\) be the defining equation of \(\Sec(S)\). The Jacobian ideal
of \(f\) is the ideal sheaf of the singular locus \(S\) in \(\Sec(S)\). 
Then the blow-up of \(\Sec(S)\) along \(S\) is the graph of the Gauss 
map~\eqref{eqs:Gauss-map}, which is the conormal variety of 
\(\Sec(S)\). 

By restricting to the hyperplane \(H\), the blow-up of \(Y\) along 
\(V\) is again the graph of the Gauss map for \(Y\). 
Corollary~\ref{cor:dual-hypsec-severi} shows that the
dual variety \(\Sec(V)^{\perp}\) is the nonsingular subvariety 
\(\pi_{[H]}(S)\subset \dualproj\). By the reflexive theorem, the 
conormal variety \(\tilde{Y}\) is isomorphic to the projective normal 
bundle of the closed subvariety \(\pi_{[H]}(S)\subset \dualproj\). Hence
\(\tilde{Y}\) is nonsingular.

Let \(X\) denote \(\Sec(S)\). Let \(\tilde{X}\) be the blow-up of 
\(X\) along \(S\), and let \(E'_0\) be the exceptional divisor in 
the blow-up \(\tilde{X}\). We claim that
\[
E_0 = E'_0\times_S V.
\] 
Let \(\mathcal{I}\subset \mathcal{O}_X\) be the ideal sheaf of \(S\),
and let \(\mathcal{J}\subset \mathcal{O}_Y\) be the ideal sheaf of 
\(V\). Since \(Y\) is a hyperplane section of \(X\), the ideal 
\(\mathcal{J}\) is the image ideal of \(\mathcal{I}\) under the 
quotient map \(\mathcal{O}_X\twoheadrightarrow \mathcal{O}_Y\).
The following surjection of graded \(\mathcal{O}_S\)-algebras
induces the inclusion \(E_0\subset E'_0\times_S V\):
\[
(\bigoplus_{q\geq 0} \mathcal{I}^q/\mathcal{I}^{q+1})\otimes_{\mathcal{O}_S} \mathcal{O}_V 
\twoheadrightarrow \bigoplus_{q\geq 0} \mathcal{J}^q/\mathcal{J}^{q+1}.
\]
Since \(E_0\) is an effective Cartier divisor in the nonsingular
variety \(\tilde{Y}\), we have
\[
\dim E_0=\dim \tilde{Y}-1=\dim E'_0-1
\]
We have proved in~\cite[Corollary 3.2]{Lyu-Zheng-secant} that \(E'_0\) 
is a smooth quadric bundle of relative dimension \(\frac{d+1}{2}\)
over \(S\). Therefore \(E'_0\times_{S} V\) is a smooth quadric 
bundle over \(V\), irreducible of dimension \(\dim E'_0-1\).
It follows that
\[
E_0 = E'_0\times_{S} V
\] 
which is a smooth quadric bundle of relative dimension \(\frac{d+1}{2}\).
\end{proof}

\subsection{Degeneration of cubic hypersurfaces}

Let \(V\subset \mathbb{P}^{n+1}\) be a general hyperplane section of
a Severi variety. Let \(X_{\infty}\subset \mathbb{P}^{n+1}\) be a 
general smooth cubic hypersurface that transversely intersects the 
smooth locus of \(\Sec(V)\) and \(V\). Let 
\(p \colon \mathcal{X}\to \Delta\) be a one-parameter degeneration of 
\(n\)-dimensional cubic hypersurfaces defined by the pencil generated 
by \(X_{\infty}\) and the central fiber \(X_0\coloneqq \Sec(V)\):
\begin{equation}
\label{eqs:one-parameter-deg}
    \mathcal{X} = \{(x,t)\in \mathbb{P}^{n+1}\times \Delta~|~F(x)+tG(x)=0\}
\end{equation}
where \(F\) and \(G\) are defining equations for \(X_0\) and
\(X_\infty\), respectively. The transversality assumptions imply
that \(X_t\) is smooth for all sufficiently small \(t\ne0\). However,
the total space \(\mathcal{X}\) is singular and the central fiber 
\(X_0\) is not a simple normal crossings divisor. In this subsection,
we provide an explicit semistable reduction of \(\mathcal{X}\to \Delta\).
We then compute the limit mixed Hodge structure associated with the
semistable degeneration.

\begin{proposition}
\label{prop:semi-red}
    A semistable model \(\mathfrak{X}\to \Delta\) of the degeneration 
    \(p \colon \mathcal{X}\to \Delta\) can be obtained by making the base
    change \(t=s^2\) and then blowing up along the subvariety \(V\)
    in the central fiber \(X_0\). The central fiber of the semistable 
    degeneration \(\pi \colon \mathfrak{X}\to \Delta\) is a simple 
    normal crossings divisor with two irreducible components 
    \(D_1\) and \(D_2\) such that
    \begin{itemize}
        \item \(D_1\) is the blow-up of \(\Sec(V)\) along \(V\);
        \item \(D_2\) is a quadric fibration over \(V\) with the
        discriminant divisor \(B = V\cap X_{\infty}\).
    \end{itemize}
    The intersection \(D_1\cap D_2\) is the exceptional divisor of 
    the blow-up \(D_1\).
\end{proposition}
\begin{proof}
The semistable reduction for secant cubics of Severi varieties was
constructed in our earlier paper~\cite{Lyu-Zheng-secant}. The same
argument applies to the hyperplane sections considered here. We recall
the main steps and refer to~\cite[Prop. 3.3]{Lyu-Zheng-secant} for the
detailed local calculation.

Suppose that \(f\) and \(g\) are local equations of the special 
hypersurface \(X_0\) and the smoothing hypersurface \(X_{\infty}\). 
Under the base change \(t=s^2\), the total space \(\mathcal{X}'\) 
defined by \(f+s^2g=0\) has double points along \(V\times \{0\}\). 
The blow-up of \(\mathcal{X}'\) along \(V\times \{0\}\) yields a smooth 
total space \(\mathfrak{X}\), and the exceptional divisor \(E\) in 
\(\mathfrak{X}\) is a quadric fibration over \(V\); this is the component 
\(D_2\) of the central fiber \(\pi\colon \mathfrak{X}\to \Delta\). 
The other component \(D_1\) is the blow-up \(\widetilde{X}_0\) of the 
secant variety \(X_0\) along \(V\), and the intersection 
\(E\cap \widetilde{X}_0\) is the exceptional divisor \(E_0\) in the
blow-up \(\widetilde{X}_0\). Lemma~\ref{lem:sm-quad-bdl} shows that
\(\widetilde{X}_0\) and \(E_0\) are nonsingular.

The total space \(E\) is smooth, although the quadric fibration
\(E\to V\) has singular fibers. Let \(u\) denote the homogeneous
coordinate associated with the additional \(\mathcal O_V\)-summand in
the exceptional projective bundle. The local normal form of \(E\) is
\begin{equation}
\label{eqs:local-normal-form}
q+u^2z=0.
\end{equation}
Here \(q\) is a family of quadratic forms parametrized by \(V\), and
\(z=0\) is a local equation of the divisor \(V\cap X_{\infty}\). The
quadratic form \(q\) is nondegenerate since it is the local 
equation of the smooth quadric bundle \(E_0\) over \(V\). Therefore 
the local normal form \(q+u^2z\) implies that the discriminant 
locus of \(E\) is the divisor \(B = V\cap X_{\infty}\).
\end{proof}

In the main theorem, we determine the limit mixed Hodge structure 
for this degeneration of cubic hypersurfaces.
The \((d-1)\)-th vanishing cohomology of \(B\) is defined as
\begin{equation}
\label{eqs:van-coh}
    \mathrm{H}^{d-1}(B; \mathbb Q)_{\mathrm{van}}\coloneqq 
    \mathrm{Ker}\left(i_*\colon \mathrm{H}^{d-1}(B; \mathbb Q)\to \mathrm{H}^{d+1}(V; \mathbb Q)(1)\right)
\end{equation}
where \(i\colon B\hookrightarrow V\) is the inclusion.
The space \(\mathrm{H}^{d-1}(B; \mathbb Q)_{\mathrm{van}}\) carries 
the polarized Hodge sub-structure that inherits from the primitive 
cohomology \(\mathrm{H}^{d-1}(B; \mathbb Q)_{\mathrm{prim}}\) with
respect to the restricted K\"ahler class
\(i^*\omega\in\mathrm H^2(B;\mathbb Q)\), where
\(\omega=c_1(\mathcal{O}_V(B))\in \mathrm{H}^2(V;\mathbb{Q})\).

\begin{theorem}
\label{thm:main-limit-mhs}
Let \(V\) be a general hyperplane section of a Severi variety as in
Subsection~\ref{subsec:severi-hyperplane}. Suppose that \(d=\dim V \in \{3, 7, 15\}\). Let 
\(X_0=\Sec(V)\) be the cubic hypersurface of secant type of dimension
\(n=\frac{3}{2}(d+1)\in \{6, 12, 24\}\).

Then the limit mixed Hodge structure \(\mathbb{V}^n_{\lim}\) associated 
with the degeneration of \(n\)-dimensional cubic hypersurfaces in~\eqref{eqs:one-parameter-deg}
is pure of weight \(n\). More precisely,
\(\mathbb{V}^n_{\lim}\) is isomorphic to
    \[
    \mathrm{H}^{d-1}(B; \mathbb{Q})_{\mathrm{van}}(-m)\oplus \mathrm{IH}^n(X_0; \mathbb{Q})_{v},
    \]
where \(m=\frac{d+5}{4}\in \{2, 3,5\}\) and 
\(\mathrm{IH}^n(X_0)_v\cong \mathbb{Q}(-\frac{n}{2})^{\oplus k}\) with 
\(k =\dim \mathrm{IH}^n(X_0;\mathbb{Q})_v\). 
\end{theorem}

The proof uses the monodromy weight spectral sequence of the
semistable model \(\pi\colon\mathfrak{X}\to\Delta\) from
Theorem~\ref{thm:mono-wt-fil}. To describe \(\mathbb{V}^n_{\lim}\)
explicitly, we compute the cohomology of the components of its central
fiber and their intersection. We give the full proof in
Section~\ref{sec:proof-main}.

\section{Decomposition of quadric fibrations}
\label{sec:quadric-fibrations}
\subsection{Decomposition theorem of perverse sheaves}

Let \(X\) be a complex algebraic variety. We denote by \(D^b_c(X; \mathbb Q)\)
the bounded derived category of constructible complexes of 
\(\mathbb{Q}\)-vector spaces on \(X\). The \emph{support} and 
\emph{cosupport conditions} define two full subcategories 
\(\prescript{p}{} D^{\leq 0}\) and \(\prescript{p}{} D^{\geq 0}\) 
respectively. \((\prescript{p}{} D^{\leq 0},~ \prescript{p}{} D^{\geq 0})\) 
forms the perverse t-structure on \(D^b_c(X)\). A complex 
\(F^{\bullet}\in D^b_c(X)\) is called a perverse sheaf if 
\[
    F^{\bullet}\in \mathrm{Perv}(X)\coloneqq \prescript{p}{} D^{\leq 0}\cap \prescript{p}{} D^{\geq 0}.
\]
The heart \(\mathrm{Perv}(X)\) of the t-structure is an abelian 
category of perverse sheaves. There are perverse truncation functors
\[
    \prescript{p}{}{\tau}_{\leq 0}\colon D^b_c(X)\to \prescript{p}{} D^{\leq 0},~ \prescript{p}{}{\tau}_{\geq 0}\colon D^b_c(X)\to \prescript{p}{} D^{\geq 0} 
\]
that are adjoint to the inclusion functors. The \(i\)-th perverse 
cohomology of \(F^{\bullet}\in D^b_c(X)\) is defined to be 
\[
    \percoh^i(F^{\bullet}) \coloneqq \prescript{p}{}{\tau}_{\leq 0}\prescript{p}{}{\tau}_{\geq 0}(F^{\bullet}[i]).
\]
In particular, if \(F^{\bullet}\) is perverse, we have
\(\percoh^i(F^{\bullet})=0, \forall i\neq 0\).

For a complex algebraic variety \(X\) of pure dimension \(n\), if \(L\) 
is a \(\mathbb{Q}\)-local system on the smooth locus \(U\), the 
\emph{(twisted) intersection complex} \(\IC_X(L)\) is a perverse sheaf
on \(X\), defined as \(j_{!*}(L[n])\)
with respect to a fixed Whitney stratification on \(X\), see~\cite{Maxim-perv-book}. In particular, \(\IC_X=\IC_X(\mathbb{Q}_{U}[n])\).

Beĭlinson, Bernstein, Deligne, and Gabber~\cite{BBDG82} established
the following decomposition of perverse sheaves.
\begin{theorem}\cite[Theorem 1.6.1]{decomp-per-sheaf}
Let \(f\colon X\to Y\) be a proper morphism of complex algebraic varieties,
where \(X\) is pure-dimensional. Then there is a non-canonical
isomorphism in \(D^b_c(Y;\mathbb Q)\)
\begin{equation}
\label{eqs:BBDG-decomp}
    Rf_*\IC_X\simeq \bigoplus_{i\in \mathbb{Z}} \percoh^i(Rf_*\IC_X)[-i].
\end{equation}
Each perverse cohomology sheaf \(\percoh^i(Rf_*\IC_X)\) is semisimple. 
More explicitly, there is a finite stratification \(Y=\bigsqcup S_\alpha\) 
by smooth locally closed subvarieties and semisimple local systems 
\(L_{\alpha,\beta}\) on \(S_\alpha\) such that
\[
\percoh^i(Rf_*\IC_X) \cong
\bigoplus_{\alpha,\beta}\IC_{\overline{S_\alpha}}(L_{\alpha,\beta}).
\]
\end{theorem}

\begin{defexam}
\label{def-exam:dec-quad}
Let \(X\) and \(Y\) be smooth, pure-dimensional algebraic varieties.
Suppose that \(g\colon X\to Y\) is a smooth projective family of
even-dimensional quadrics. Set \(n=\dim X\) and let \(r\) be the
fiber dimension. Deligne's canonical decomposition gives
\begin{equation}
    \label{eqs:deligne-decomp-thm}
        Rg_*\mathbb{Q}_X[n]\simeq \bigoplus_{i=-r}^{r} R^{r+i}g_*\mathbb{Q}_{X}[\dim Y-i].
    \end{equation}
Under this splitting, the \(i\)-th perverse cohomology sheaf
is represented by
\[
 \percoh^i(Rg_*\mathbb{Q}_X[n])\simeq R^{r+i}g_*\mathbb{Q}_{X}[\dim Y]. 
\]

Let \(p\colon \mathbb P(\mathcal F)\to Y\) be a projective bundle 
such that \(j: X\hookrightarrow \mathbb P(\mathcal F)\) is cut out by 
a section of the line bundle 
\(\mathcal O_{\mathcal F}(2)\otimes p^*\mathcal A\). We denote by 
\(\eta\in \Gamma(Y, R^2g_*\mathbb{Q})\) the global section induced by 
\(c_1(O_{X}(1))\). For \(0\leq i\leq r\), the local system 
\(R^{2i}g_*\mathbb Q_X\) is generated by the global section \(\eta^i\) 
except \(i=\frac{r}{2}\). The middle cohomology local system 
\(R^{r}g_*\mathbb{Q}\) admits the following canonical splitting
\begin{equation}
\label{eqs:anti-inv-loc-sys-quad}
R^{r}g_*\mathbb{Q}\cong \mathbb{Q}_Y\cdot \eta^{\frac{r}{2}}\oplus \mathbb{L}
\end{equation} 
where \(\mathbb{L}\) is the local system of the primitive part
\begin{equation}
\label{eqs:primitive-loc-sys}
    \mathbb{L}\coloneqq \mathrm{Ker}(j_*\colon R^rg_*\mathbb Q_X\to R^{r+2}p_*\mathbb Q_{\mathbb P(\mathcal F)}).
\end{equation}
The stalks of \(\mathbb L\) are spanned by the difference of the 
cohomology classes of two maximal isotropic subspaces in the fibers.


We denote by \(\IC_Y^{+}\) the perverse sheaf \(\mathbb{Q}_Y[\dim Y]\) 
(we omit \(+\) when no confusion is likely), and \(\IC_Y(\mathbb L)\)
the perverse sheaf \(\mathbb{L}[\dim Y]\).
\end{defexam}

\subsection{Decomposition formula for quadric fibrations}
\label{subsec:decomp-quad-fib}
We now apply the decomposition theorem to quadric fibrations.
Fix the following setting:

\begin{situation}
\label{setting:quad-fib}
Let \(r > 0\) be an even integer. Let \(X\) and \(V\) be smooth complex
projective varieties. Suppose that \(f: X\to V\) is a quadric fibration of 
relative dimension \(r+1\), and the discriminant locus of \(f\) is a smooth 
divisor \(B\subset V\) such that the fiber \(f^{-1}(b)\) for any 
\(b\in B\) is a quadric cone with a single vertex.         
\end{situation}

\begin{definition}
\label{def:rel-var-coh}
Let \(p\colon \mathbb P(\mathcal E)\to V\) be a projective bundle of
relative dimension \(r+2\) such that \(X\) is a divisor of the
line bundle \(\mathcal{O}_{\mathbb P(\mathcal E)}(2)\otimes p^*\mathcal{A}\) 
for some \(\mathcal{A}\in \mathrm{Pic}(V)\)
(cf. Beauville~\cite[Prop. 1.2]{Bea-quadric-prmy-77}). One defines the relative 
variable cohomology sheaf to be
    \begin{equation}
    \label{eqs:var-coh-grp}
        (R^{r+2}f_* \mathbb{Q})_{\textrm{var}}\coloneqq \textrm{Coker} (R^{r+2}p_* \mathbb{Q}\to R^{r+2}f_* \mathbb{Q}).
    \end{equation}    
\end{definition}

This sheaf appears in the decomposition formula for \(f:X\to V\).
For each \(s\in V\setminus B\), the restriction map
    \[
    \mathrm{H}^{r+2}(\mathbb{P}^{r+2}, \mathbb{Q})\to \mathrm{H}^{r+2}(f^{-1}(s), \mathbb{Q})
    \]
is an isomorphism. Hence \((R^{r+2}f_* \mathbb{Q})_{\textrm{var}}\) is 
supported on the discriminant locus \(B\). 
\begin{lemma}
\label{lem:var-coh-loc-sys}
The relative variable cohomology sheaf \((R^{r+2}f_* \mathbb{Q})_{\textrm{var}}\)
is a local system on the discriminant divisor \(B\), isomorphic to the
primitive local system \(\mathbb L\) in the sense of~\eqref{eqs:primitive-loc-sys}.
\end{lemma}

\begin{proof}
It suffices to show that the restriction \(R^{r+2}f_* \mathbb{Q}|_B\) 
is a local system. Set \(X_B=f^{-1}(B)\); this is a family of quadric
cones over \(B\). It admits a section \(s\colon B\to X_B\) whose
value at \(t\) is the unique vertex in \(f^{-1}(t)\). The blow-up
\(\epsilon \colon \tilde{X}_B\to X_B\) along the section \(s(B)\) is 
a simultaneous resolution of all quadric cones along the vertex. Then 
\((f \circ \epsilon)^{-1}(t)\) is a \(\mathbb{P}^1\)-bundle over a 
\(r\)-dimensional smooth quadric. By contracting these \(\mathbb{P}^1\), 
we obtain a family \(h\colon Y\to B\) of \(r\)-dimensional smooth 
quadrics, which fits into the following diagram:
\begin{equation}
    \label{eqs:rel-quadric-proj}
        \begin{tikzcd}[column sep=small, row sep=small]
            & \widetilde{X}_B \ar[ld, "\epsilon"'] \ar[rd, "\alpha"] & \\
            X_B \ar[rd, "f_B"']& &  Y \ar[ld, "h"]\\
            & B & .
        \end{tikzcd}
    \end{equation}
Since \(r\) is an even integer, \(R^rh_*\mathbb{Q}_Y\) is 
a rank two local system. We claim that 
\begin{equation}
\label{eqs:loc-sys-even-dim}
   R^{r+2}{f_B}_* \mathbb{Q}\cong R^rh_*\mathbb{Q}_Y.
\end{equation}

    Let \(U\coloneqq X_B\setminus s(B)\) be the open complement of 
    the section \(s(B)\). Denote by \(j\colon U\hookrightarrow X_B\)
    the open embedding and \(f_{U}\colon U\to B\) the restricted 
    morphism. Consider the distinguished triangle of constructible 
    complexes
    \[
    j_! \mathbb{Q}_{U}\to \mathbb{Q}_{X_B}\to s_*\mathbb{Q}_{B}\overset{[1]}{\to}
    \]
    in \(D^b_c(X_B;\mathbb{Q})\). Applying \(R{f_B}_*\) gives
    \[
        R{f_{U}}_!\mathbb{Q}_{U}\to R{f_B}_*\mathbb{Q}_{X_B}\to \mathbb{Q}_{B}\overset{[1]}{\to}
    \]
    in \(D^b_c(B;\mathbb{Q})\), which implies
    \[
        R^{r+2}{f_U}_!\mathbb{Q}_{U}\simeq R^{r+2}{f_B}_*\mathbb{Q}_{X_B}.
    \]
    Away from the vertex section, the rational map
    \(\alpha \circ \epsilon^{-1}\colon X_B\dashrightarrow Y\) is
    the relative projection from \(s(B)\). This projection
    from the vertex locus \(s(B)\) makes \(U\) an \(\mathbb A^1\)-bundle over \(Y\):
    \[
        \alpha \colon U\cong \epsilon^{-1}(U) \rightarrow Y.
    \]
    Since \(\alpha\) is an \(\mathbb A^1\)-bundle, we have
    \[
    R\alpha_!\mathbb{Q}_{U} \simeq R^2\alpha_!\mathbb{Q}_{U}[-2]\simeq \mathbb{Q}_{Y}[-2].
    \]
    Applying the Grothendieck spectral sequence to the composed functor
    \(f_{U} = h \circ \alpha\) yields
    \[
        R^{r+2}{f_{U}}_!\mathbb{Q}_{U}\simeq R^{r}h_! R^2\alpha_!\mathbb{Q}\simeq R^{r}h_*\mathbb{Q}_{Y}
    \]
    where \(Rh_!=Rh_*\) because \(h\) is proper. This proves
    \eqref{eqs:loc-sys-even-dim}.
    
    Let \(P\) be the projective bundle in Definition~\ref{def:rel-var-coh}.    
    Let \(P_B=p^{-1}(B)\). The analogous blow-up along \(s(B)\) and contraction
    produce a projective bundle \(k\colon W\to B\) of relative
    dimension \(r+1\). They fit into the following diagram:
    \begin{equation}
    \label{pic:blown-up-down}
\begin{tikzcd}[column sep=normal, row sep=0.5em]
  & \widetilde{X}_B
      \arrow[dl, "\varepsilon"']
      \arrow[dd, hook]
      \arrow[dr, "\alpha"]
  & \\
  X_B
      \arrow[dd, hook]
      \arrow[dddr, dashed, "f_B"{pos=0.25}]
  &
  & Y
      \arrow[dd, hook]
      \arrow[dddl, dashed, "h"'{pos=0.25}] \\
  & \widetilde{P}_B
      \arrow[dl, crossing over]
      \arrow[dr, crossing over]
  & \\
  P_B
      \arrow[dr, "p_B"']
  &
  & W
      \arrow[dl, "k"] \\
  & B &
\end{tikzcd}
\end{equation}
    The same argument gives the isomorphisms
    \[
    R^{r+2}{p_B}_*\mathbb{Q}_{P} \simeq R^{r+2}{p_U}_!\mathbb{Q}_{U_P}\simeq R^{r}k_*\mathbb{Q}_{W}.
    \]
    Here \(p_B\colon P_B\to B\) is the projection map, and \(U_P\)
    is the open subset \(P_B\setminus s(B)\). The projection 
    \(p_U\colon U_P\to B\) factors through an \(\mathbb{A}^1\)-bundle 
    map \(U_P\to B\). These isomorphisms fit into the following 
    commutative diagrams
    \[
        \begin{tikzcd}[column sep=normal, row sep=small]
        R^{r+2}{p_B}_*\mathbb{Q} \ar[d] & R^{r+2}{p_U}_!\mathbb{Q}\ar[l, "\sim"'] \ar[r, "\sim"] \ar[d] & R^{r}k_*\mathbb{Q} \ar[d]\\
        R^{r+2}{f_B}_*\mathbb{Q}   & R^{r+2}{f_U}_!\mathbb{Q} \ar[l, "\sim"'] \ar[r, "\sim"] & R^{r}h_*\mathbb{Q}
        \end{tikzcd}
    \]
    where the cokernel of the leftmost vertical map is 
    \((R^{r+2}f_* \mathbb{Q})_{\textrm{var}}\). The cokernel of the rightmost
    vertical map is naturally isomorphic to the primitive local system
    \(\mathbb L\) in~\eqref{eqs:primitive-loc-sys}. Thus
    \[
    (R^{r+2}f_* \mathbb{Q})_{\textrm{var}}\cong \mathbb{L}.
    \]
\end{proof}

\begin{theorem}
\label{thm:perv-decomp-quad-fib}
Let \(f\colon X\to V\) be the quadric fibration in Situation~\ref{setting:quad-fib}. Set \(n=\dim X\).
The decomposition theorem gives
\begin{equation}
\label{eqs:perv-decomp-quad-fib}
    Rf_*\IC_{X}\simeq \bigoplus_{j=0}^{r+1} \IC_{V}[r+1-2j]
    \oplus \IC_B(L),
\end{equation}
where \(L\) is the variable cohomology local system~\eqref{eqs:var-coh-grp}.
\end{theorem}

\begin{proof}
    Let \(U\coloneqq V\setminus B\) be the open subset on which the 
    restricted morphism \(f_U\colon X_U\to U\) is a smooth quadric 
    bundle of odd relative dimension \(r+1\). Deligne's
    Hard Lefschetz theorem gives the canonical decomposition
    \[
        R{f_U}_*\mathbb{Q}_{X_U}[n]\simeq 
        \bigoplus_{\substack{i\in [-r-1, r+1]\\ i~ \textrm{odd}}}R^{i+r+1}{f_U}_*\mathbb{Q}_{X_U}[\dim V-i]
    \]
    The intersection complex \(\IC_V\) is the unique intermediate 
    extension of the perverse sheaf 
    \(R^{i+r+1}{f_U}_*\mathbb{Q}_{X_U}[\dim V]\simeq \mathbb{Q}_U[\dim V]\). 
    Therefore we have
    \begin{equation}
    \label{eqs:decomp-quad-fib}
        Rf_*\IC_{X}\simeq \bigoplus_{\substack{i\in [-r-1, r+1]\\ i~ \textrm{odd}}} \IC_{V}[-i]\oplus \bigoplus_{\alpha, j} \IC_{\overline{S_\alpha}}(L_{\alpha, j})[-j],
    \end{equation}
    where each stratum \(S_{\alpha}\) is contained in the discriminant
    locus \(B\), and \(L_{\alpha,j}\) is a rational local
    system on \(S_{\alpha}\).

    We claim that only strata of dimension \(\dim B\) can occur in
    \eqref{eqs:decomp-quad-fib}, and only in perverse degree \(j=0\).
    Since \(f\) is a projective morphism, we have the isomorphism
    \[
        \mathcal{H}^k_{s}(Rf_*\mathbb{Q}_{X}[n])\simeq \mathrm{H}^{k+n}(f^{-1}(s), \mathbb{Q})
    \]
    of \(k\)-th cohomology stalk at any \(s\in B\). Note that \(f^{-1}(s)\)
    is a quadric cone with a single vertex. Then 
    \(\dim \mathrm{H}^{k+n}(f^{-1}(s), \mathbb{Q})=1\) if and only if 
    \(k+n\) is an even integer in \([0, 2r+2]\) and \(k+n \neq r+2\). 
    For every such \(k\), there is an odd integer \(i = k+n-r-1\in [-r-1, r+1]\)
    such that \(\IC_V[-i]\) is a direct summand in the 
    decomposition~\eqref{eqs:decomp-quad-fib} and 
    \(\mathcal{H}^{k}_s(\IC_V[-i])\simeq \mathbb{Q}\). Then the 
    direct summand \(\IC_{\overline{S_\alpha}}(L_{\alpha, j})[-j]\) 
    can have a nonzero \(k\)-th cohomology stalk at \(s\in B\)
    only when \(k = r+2-n = -\dim B\). Moreover,
    \[
        \mathcal{H}^k_s(\IC_{\overline{S_\alpha}}(L_{\alpha, j})[-j])=(L_{\alpha, j})_s
    \] 
    holds for \(k+\dim S_{\alpha}-j=0\) and \(s\in S_{\alpha}\subset B\). 
    It follows that \(\dim S_{\alpha}-j=\dim B\).

    It is known that \(Rf_*\IC_{X}\) is Verdier self-dual and
    \(D_V(\IC_{V}[-i]) \simeq \IC_{V}[i]\). Hence the direct sum 
    \(\bigoplus \IC_{\overline{S_\alpha}}(L_{\alpha, j})[-j]\) is 
    self-dual. Therefore, the Verdier dual
    \[
        D_V(\IC_{\overline{S_\alpha}}(L_{\alpha, j})[-j])\simeq \IC_{\overline{S_\alpha}}(L^{\vee}_{\alpha, j})[j]
    \]
    is also a direct summand that appears among 
    \(\IC_{\overline{S_\alpha}}(L_{\alpha, j})[-j]\). By the argument 
    in the previous paragraph, it follows that 
    \(\dim S_{\alpha}+j=\dim B\), and hence \(\dim S_{\alpha}=\dim B\)
    and \(j=0\). Since \(B\) is irreducible and pure-dimensional,
    the stratum \(S_{\alpha}=B\) is unique. Its intersection complex is
    \[
    \IC_B(L)=i_*L[\dim B], \quad L ~\textrm{is a local system on}~ B,
    \] 
    where \(i\colon B\hookrightarrow V\) is the inclusion. 
    As perverse cohomology sheaves, we have
    \[
    \begin{split}
        \percoh^i(Rf_*\IC_{X}) & \simeq  \IC_{V}, \quad \forall ~\textrm{odd}~ i\in [-r-1,r+1],\\
        \percoh^0(Rf_*\IC_{X}) & \simeq \IC_B(L).
    \end{split}
    \]
    After reindexing \(i\) by \(r+1-2j\), we obtain the decomposition
    formula~\eqref{eqs:perv-decomp-quad-fib}.

    Now let us identify the local system \(L\). Let \(p\colon P\to V\) 
    be the projective bundle in Definition~\ref{def:rel-var-coh}. 
    Consider the restriction morphism
    \[
        Rp_*\mathbb Q_P[n]\to Rf_* \mathbb Q_{X}[n].
    \]
    The intersection complex \(\IC_B(L)\) is concentrated in degree
    \(-n+r+2\). Taking ordinary cohomology in this degree gives
    \[
    R^{r+2}p_* \mathbb{Q}\to R^{r+2}f_* \mathbb{Q}=\mathbb Q_{V}\oplus i_*L,
    \]
    The image of \(R^{r+2}p_* \mathbb{Q}\) is \(\mathbb Q_V\), and
    \(L\) is isomorphic to the variable-cohomology local system
    \((R^{r+2}f_* \mathbb{Q})_{\textrm{var}}\) in~\eqref{eqs:var-coh-grp}.
    We also note the duality
    \[
    L\simeq L^{\vee}(-r-2)
    \] 
    which follows from
    \[
    \percoh^0(Rf_*\IC_{X})\simeq (\percoh^0(Rf_*\IC_{X}))^{\vee};
    \]
    see~\cite[Theorem 3.3.1]{dCAM05}.
\end{proof}

Let \(E\) (resp. \(E_0\)) be the quadric fibration (resp. smooth
quadric bundle) in Proposition~\ref{prop:semi-red}. The quadric 
fibration \(f:E\to V\) satisfies Situation~\ref{setting:quad-fib};
here \(V\) and \(B\) are the varieties in Proposition~\ref{prop:semi-red}.
Example~\ref{def-exam:dec-quad} and Theorem~\ref{thm:perv-decomp-quad-fib}
give the following decompositions of constructible complexes:
\[
Rf_*\mathbb Q_E \simeq \bigoplus_{j=0}^{r+1}\mathbb Q_V[-2j]\oplus i_*L[-r-2],\quad
Rg_*\mathbb Q_{E_0} \simeq \bigoplus_{j=0}^{r}\mathbb Q_V[-2j]\oplus \mathbb L[-r]
\]

In our setting, the base spaces \(V\) and \(B\) are simply connected, 
since both are rational homogeneous spaces. Consequently, the local 
systems \(\mathbb L\)  and \(L\) are the constant local systems
\(\mathbb Q_V\) and \(\mathbb Q_B\), respectively. In the rest of the
paper, we keep the notation \(\mathbb L\) and \(L\) in order to 
distinguish these summands from those generated by the hyperplane 
section classes.

\begin{remark}
\label{rmk:pullback-prim-loc-sys}
We have shown that \(L\) is isomorphic to the primitive part of
\(R^rh_*\mathbb Q\) for a smooth quadric bundle
\(h\colon Y\to B\). In the present setting, \(Y\) is
identified with the restriction of \(E_0\) to \(B\), and there is a
natural isomorphism of local systems 
\begin{equation}
\label{eqs:iso-two-prim-loc}
i^*\mathbb L\xrightarrow{\sim} L, \quad i\colon B\hookrightarrow V.
\end{equation} 
In fact, at any point \(b\in B\), the stalk map \(\mathbb L_b\to L_b\)
is the restriction to the primitive parts of the Thom isomorphism
\[
\mathrm H^{r}(E_{0,b}) \xrightarrow{\sim} \mathrm H^{r+2}(E_b).
\]
This map can be globally defined as follows.

In the present situation, the projective bundle \(P\) in
Definition~\ref{def:rel-var-coh} is
\(\mathbb{P}(N_{V/\mathbb{P}^{n+1}_{\Delta}})\) where
\(N_{V/\mathbb{P}^{n+1}_{\Delta}}\) is the normal bundle of the
subvariety \(V\) in the central fiber of the total space 
\(\mathbb{P}^{n+1}_\Delta\), cf.~\eqref{eqs:one-parameter-deg}. 
Let \(P_0\) be the projective normal bundle 
\(\mathbb{P}(N_{V/\mathbb{P}^{n+1}})\), which contains \(E_0\)
as a relative divisor of degree two. The natural splitting
\begin{equation}
\label{eqs:normal-bdl-split}
N_{V/\mathbb{P}^{n+1}_{\Delta}}\cong N_{V/\mathbb{P}^{n+1}}\oplus \mathcal{O}_V
\end{equation}
induces the projection \(N_{V/\mathbb{P}^{n+1}_{\Delta}}\to \mathcal{O}_V\),
which defines a section of \(p\colon P\to V\). Recall the local normal 
form \(q+u^2z\) in~\eqref{eqs:local-normal-form} of \(E\). Over a point 
\(b\in B=\{z=0\}\), the singular vertex of the quadric cone fiber 
\(E_b=\{q_b(w)=0\}\subset \mathbb P(w\colon u)\) is defined by \(w=0, u=1\). 
Therefore the section \(s\colon B\to E|_B\) of the singular vertices 
is induced by the quotient \(N_{V/\mathbb{P}^{n+1}_{\Delta}}\to\mathcal{O}_V\).
It follows from the splitting~\eqref{eqs:normal-bdl-split} that the 
projective bundle \(W\) in the diagram~\eqref{pic:blown-up-down} is 
identified with \(P_0|_B\). Since \(E_0\) does not meet the section \(s(B)\) 
in \(E|_B\), the projection \(\alpha\colon \widetilde{E}|_B\to Y\) 
induces an isomorphism \(\alpha\colon E_0|_B\to Y\) of quadric bundles
over \(B\). This gives the desired isomorphism
\(i^*\mathbb L\xrightarrow{\sim}L\).
\end{remark}

\subsection{Trace-orthogonality}
\label{subsubsec:trace-orthogonal}
\begin{proposition}
\label{prop:trace-orthognal}
After choosing suitable split monomorphisms
\[
i_*L[-r-2]\to Rf_*\mathbb Q_E, \quad \mathbb L[-r]\to Rg_*\mathbb Q_{E_0},
\]
the primitive summands \(L[-r-2]\)  and \(\mathbb L[-r]\) are orthogonal
to the hyperplane summands in their decompositions:
\begin{equation}
\label{eqs:trace-orthogonal-decomp}
\left(\bigoplus_{j=0}^{r+1}\mathbb Q_V[-2j]\right)
\perp i_*L[-r-2],\qquad
\left(\bigoplus_{j=0}^{r}\mathbb Q_V[-2j]\right)
\perp\mathbb L[-r].
\end{equation}
with respect to the trace pairings
\[
\begin{aligned}
&\mathrm{Tr}_f\colon Rf_*\mathbb Q_E\overset{L}{\otimes} Rf_*\mathbb Q_E \to \mathbb Q_V(-r-1)[-2r-2];\\
&\mathrm{Tr}_g\colon Rg_*\mathbb Q_{E_0}\overset{L}{\otimes} Rg_*\mathbb Q_{E_0}\to \mathbb Q_V(-r)[-2r].
\end{aligned}
\]
\end{proposition}
\begin{proof}
Denote by \(\eta_0=c_1(\mathcal O_{E_0}(1))\) the hyperplane class of
\(E_0\). We claim that there is a cohomology class
\(\xi\in\mathrm H^r(E_0)\) that maps to a generator section of
\(\mathbb L\) under the Leray spectral sequence map
\(\mathrm H^r(E_0)\to\mathrm H^0(V,\mathbb L)\) and satisfies
\[
\mathrm{Tr}_g(\xi\cup \eta_0^k)=0, \quad 0\leq k\leq r.
\]
Let \(p\colon P_0\to V\) be the projective bundle in Definition~\ref{def-exam:dec-quad}, 
and \(j\colon E_0\hookrightarrow P_0\) be the inclusion.
By the projective bundle formula, we can write
\[
j_*\xi=\sum_{i=0}^{\frac{r}{2}}p^*\beta_{r+2-2i}\cdot h_0^i,
\qquad
\beta_{r+2-2i}\in \mathrm H^{r+2-2i}(V,\mathbb Q).
\]
where \(h_0=c_1(\mathcal O_{P_0}(1))\). The coefficient of 
\(h_0^{\frac{r}{2}+1}\) is zero because \(j_*\xi\) restricts to zero 
on every fiber of \(P_0\). The divisor class of \([E_0]\) in \(P_0\) 
is \(2h_0+p^*a\) where \(a=c_1(\mathcal{A})\), see Definition~\ref{def-exam:dec-quad}. 
Polynomial division gives
\[
j_*\xi=(2h_0+p^*a)\cdot \gamma+p^*\beta,
\]
where
\[
\gamma=\sum_{i=0}^{\frac{r}{2}-1} p^*\gamma_{r-2-2i}\cdot h_0^i \in \mathrm{H}^{r}(P_0), \qquad
\beta=\sum_{i=0}^{\frac{r}{2}}\left(-\frac{a}{2}\right)^i\beta_i.
\]
The class \(\gamma_{r-2-2i}\) has positive degree, so \(\gamma\) restricts
to zero on each fiber of \(E_0\). Replace \(\xi\) by 
\(\xi'=\xi-j^*\gamma\). Then \(\xi'\) is a class in \(\mathrm H^{r}(E_0)\)
that sends to a section of \(\mathbb L\) and
\[
    j_*\xi'=j_*\xi-j_*j^*\gamma=j_*\xi-\gamma\cdot[E_0]=p^*\beta.
\] 
Then the trace map adjunction yields
\[
    \mathrm{Tr}_g(\xi'\cup \eta_0^k)=\mathrm{Tr}_p(j_*\xi'\cup h_0^k)=\beta\cup p_*h_0^k=0, \quad 0\leq k\leq r.
\]

For the local system \(L\), likewise choose a cohomology class
\(\zeta\in \mathrm H^{r+2}(E)\) that maps to a generator of
\(\mathrm H^0(B, L)\) and 
\[
    \mathrm{Tr}_f(\zeta\cup \eta^k)=0,\quad 0\le k\le r+1.
\]
The argument parallels the \(E_0\) case. Let \(p\colon P\to V\)
be the projective bundle in Definition~\ref{def:rel-var-coh}, and
\(j\colon E\hookrightarrow P\) be the inclusion. We have
\[
    j_*\zeta=\sum_{j=0}^{\frac{r}{2}+1}p^*i_*\alpha_{r+2-2j}\cdot h^j,
\]
where \(h=c_1(\mathcal O_P(1))\), and \(\alpha_{r+2-2i}\) is the class 
determined by the evaluation \((\zeta, \phi_j)\in \mathrm H^{r+2-2j}(B)\) with
\[
\phi_j\in \mathrm{Hom}(i_*L[-r-2], \mathbb{Q}_V[2-2j])\cong  \mathrm{Hom}(L, i^!\mathbb{Q}_V[r+4-2j])=\mathrm H^{r+2-2j}(B, L^{\vee}).
\]
Polynomial division gives a class
\(\zeta'\in \mathrm H^{r+2}(E)\) that lifts the generator section of \(L\) and satisfies
\[
    j_*\zeta'=p^*i_*\alpha, \quad \alpha\in \mathrm H^{r+2}(B).
\]
The trace map adjunction yields
\[
\mathrm{Tr}_f(\zeta'\cup \eta^k)=\mathrm{Tr}_p(j_*\zeta'\cup h^k)=i_*\alpha\cup p_*h^k=0, \quad 0\leq k\leq r+1.    
\]
\end{proof}

\subsection{Blow-up decomposition}
\begin{proposition}
\label{prop:decomp-blow-up}
	Let \(Y\) be an \(n\)-dimensional complex algebraic variety. Assume
	that \(Z\subset Y\) is a smooth subvariety of codimension \(c+1\)
    such that \(Y\) has transverse \(A_1\) singularities along \(Z\) and
    \(U=Y\setminus Z\) is the smooth locus of \(Y\). Let
	\(\tau\colon \widetilde{Y}\to Y\) be the blow-up of \(Y\) along \(Z\). 
    Then
	\begin{equation}
    \label{eqs:decomp-blow-up}
		R\tau_*\IC_{\widetilde{Y}}\simeq \IC_Y\oplus \bigoplus_{j=1}^{c} \IC_Z[c+1-2j]. 
	\end{equation}
\end{proposition}
\begin{proof}
    Let \(D\) be the exceptional divisor in the blow-up \(\widetilde{Y}\).
    Since \(Y\) has transverse \(A_1\) singularities along \(Z\), the
    projection map \(g\colon D\to Z\) is a smooth quadric bundle of 
    relative dimension \(c\).

    We have \((R\tau_*\IC_{\widetilde{Y}})|_U = \IC_U\). Since \(\IC_U\)
    uniquely extends to \(\IC_Y\), the complex
    \(R\tau_*\IC_{\widetilde{Y}}\) contains \(\IC_Y\) as a direct
    summand. All other summands are supported on \(Z\). Thus the
    decomposition has the form
    \begin{equation}
    \label{eqs:decomp-perv-blow-up}
    R\tau_*\IC_{\widetilde{Y}} \simeq \IC_Y\oplus \bigoplus_{\alpha, k} 
    \IC_{\overline{Z_\alpha}}(L_{\alpha,k})[d_{\alpha,k}]	
    \end{equation}
    where \(Z_{\alpha}\) are strata contained in \(Z\), \(L_{\alpha,k}\) 
    are local systems on \(Z_{\alpha}\) and \(d_{\alpha,k}\) are integers.

    By Deligne's construction of the intersection complex, we have
    \[
    	\IC_Y \cong \tau_{\leq c-n}(Rj_*\mathbb{Q}_U[n])
    \]
    where \(j\colon U\hookrightarrow Y\) is the open immersion. It 
    follows that
    \[
    	\mathcal{H}^{i}_x(\IC_Y)=0, \quad \forall ~i> c-n.
    \]
    Hence the stalks \(\mathcal{H}^i_z(R\tau_*\IC_{\widetilde{Y}})\) 
    for any \(z\in Z\) and \(i+n \geq c+1\) come from the 
    summands \(\IC_{\overline{Z_\alpha}}(L_{\alpha,k})[d_{\alpha,k}]\). 
    Note that 
    \[
    \mathcal{H}^i_z(R\tau_*\IC_{\widetilde{Y}})\cong \mathrm{H}^{i+n}(D_z)
    \] 
    where \(z\in Z\) and \(D_z\) is the fiber of \(g\colon D\to Z\). 
    If \(i+n\) is even and \(c < i+n \leq 2c\),
    the local system \(R^{i+n}g_*\mathbb{Q}_{D}\) represents the local
    system \(L_{\alpha, k}\) with \(Z_{\alpha}=Z\) and 
    \(d_{\alpha, k}=c+1-(i+n)\). Write \(2j = i+n\in [c+1, 2c]\). 
    The decomposition~\eqref{eqs:decomp-perv-blow-up} contains the 
    direct sum
    \[
    \bigoplus_{c+1-2j\in [-c+1,0] } \IC_Z[c+1-2j].
    \] 
    Comparing stalk dimensions shows that this accounts for
    all direct summands with non-positive shifts
    in~\eqref{eqs:decomp-perv-blow-up}.
    
    Because \(R\tau_*\IC_{\widetilde{Y}}\) is Verdier self-dual,
    the Verdier dual 
    \[
    \IC_Z[c+1-2j]^{\vee}=\IC_Z[-c-1+2j],\quad -c-1+2j \in [0, c-1]
    \] 
    is also a component of the decomposition. There are no other
    positively shifted summands: otherwise, Verdier duality would
    produce an additional negatively shifted summand. After reindexing
    \(j\), the direct sum
    \begin{equation}
    \label{eqs:direct-sum-center}
    	\bigoplus_{j\in [1,c]} \IC_Z[c+1-2j]
    \end{equation}
    complements \(\IC_Y\) in the decomposition~\eqref{eqs:decomp-perv-blow-up}.
\end{proof}

\begin{corollary}
\label{cor:prop-base-change}
Keep the notations in Proposition~\ref{prop:decomp-blow-up}. Suppose
also that \(c\) is even. Let \(\iota: Z\hookrightarrow Y\) be the
closed embedding, and let \(\mathbb L\) be the local system associated
with the smooth quadric bundle \(g\colon D\to Z\) as in
Definition~\ref{def-exam:dec-quad}. Then
    \[
    \iota^*\IC_Y\simeq \IC_Z[c+1]\oplus \IC_Z(\mathbb L)[1].
    \]
\end{corollary}
\begin{proof}
    Proper base change applied to the blow-up square
    \[
        \begin{tikzcd}
        D \ar[r, hook, "e"] \ar[d, "g"] & \widetilde{Y} \ar[d, "\tau"]\\
        Z \ar[r, hook, "\iota"] & Y
        \end{tikzcd}
    \]
    yields the isomorphisms
    \[
    \iota^*R\tau_*\IC_{\widetilde{Y}}\simeq Rg_*e^*\IC_{\widetilde{Y}}
    \simeq Rg_*\IC_D[1].
    \] 
    The quadric bundle \(D\) has even fiber dimension \(c\). Hence
    the formulas~\eqref{eqs:deligne-decomp-thm}
    and~\eqref{eqs:decomp-blow-up} imply the isomorphism
    \begin{equation}
    \label{eqs:proper-base-change-decomp}
    \iota^*\IC_Y\oplus \bigoplus_{j=1}^c \mathbb Q_Z[n-2j]\xrightarrow{\sim}    
    \IC_Z[c+1]\oplus \IC_Z(\mathbb L)[1]\oplus \bigoplus_{j=1}^c \mathbb Q_Z[n-2j]. 
    \end{equation}
    Let \(\mathcal O_D(1)\) be a relatively ample line bundle on \(D\),
    and \(h=c_1(\mathcal O_D(1))\). The correspondence 
    \[
    \alpha\mapsto e_*(h^j\cup g^*\alpha)
    \] 
    gives a natural split morphism 
    \(\mathbb Q_Z[n-2j]\to R\tau_*\mathbb Q_{\widetilde{Y}}[n]\).
    The self intersection formula and 
    \(N_{D/\widetilde Y}\simeq \mathcal O_D(-1)\) together yields
    \[
        e^*e_*(h^{j-1}\cup g^*\alpha)=c_1(N_{D/\widetilde Y})\cup h^{j-1}\cup g^*\alpha=-h^j\cup g^*\alpha, \quad \forall 1\leq j\leq c.
    \]
    Hence the pushforward map \(e_*\) identifies the components
    \(\mathbb{Q}_Z[n-2j]\) on the two sides of the isomorphism~\eqref{eqs:proper-base-change-decomp}.
    The remaining summands in~\eqref{eqs:proper-base-change-decomp}
    give the required isomorphism
    \[
     \iota^*\IC_Y\simeq \IC_Z[c+1]\oplus \IC_Z(\mathbb L)[1].
    \]
This isomorphism can also be obtained by examining the links associated
to the normal slice of the stratum \(Z\subset Y\), which is the Stiefel
manifold of orthonormal \(2\)-frames in \(\mathbb{R}^{c+2}\).
\end{proof}

\section{Proof of main theorem}
\label{sec:proof-main}
\begin{proof}[Proof of Theorem~\ref{thm:main-limit-mhs}]
Let \(\mathbb{V}_{\infty, \mathbb{Q}} = \mathbb{H}^{n}(D; \psi_{\pi} \mathbb{Q}_{\mathfrak{X}})\) 
be the canonical fiber of the semistable degeneration
\(\pi\colon\mathfrak{X}\to\Delta\). By Theorem~\ref{thm:mono-wt-fil},
the monodromy weight spectral sequence
    \[
	{}_W {E^{-\ell, n+\ell}_1}=\bigoplus_{p\geq \max\{0, -\ell\}} \mathrm{H}^{n-\ell-2p}(D(2p+\ell+1);\mathbb{Q})(-\ell-p)
        \Rightarrow \mathbb{V}_{\infty, \mathbb{Q}}
    \]
degenerates at \(E_2\). Hence
\[
    {}_WE^{-\ell, n+\ell}_2={}_WE^{-\ell, n+\ell}_{\infty}= \mathrm{Gr}^W_{n+\ell} \mathbb{V}_{\infty}.
\]
In our situation, the central fiber \(D\) has two components. 
Then \(E^{-\ell, n+\ell}_1=0\) for \(\ell\notin [-1, 1]\).
The \(E_1\)-terms relevant to our calculation are shown below:
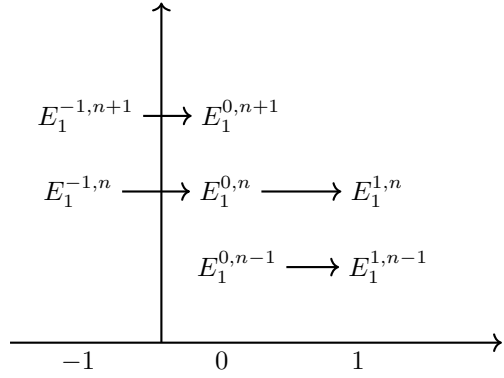
\begin{center}
\begin{tikzpicture}
    \draw[->, thick] (-2,0) -- (4.5,0) ;
    \draw[->, thick] (0,0) -- (0,4.5) ;

    \node (E-1n1) at (-1, 3) {\(E_1^{-1,n+1}\)};
    \node (E0n1) at (1.05, 3) {\(E_1^{0,n+1}\)};
    \draw[->, thick] (E-1n1) -- (E0n1);

    \node (E0n) at (0.85, 2) {\(E_1^{0,n}\)};
    \node (E1n) at (2.85, 2) {\(E_1^{1,n}\)};
    \node (E-1n) at (-1.1, 2) {\(E_1^{-1, n}\)};
    \draw[->, thick] (E-1n) -- (E0n);
    \draw[->, thick] (E0n) -- (E1n);

    \node (E0nm1) at (1, 1) {\(E_1^{0,n-1}\)};
    \node (E1nm1) at (3, 1) {\(E_1^{1,n-1}\)};
    \draw[->, thick] (E0nm1) -- (E1nm1);
    
    \node[below] at (-1.1,0) {\(-1\)};
    \node[below] at (0.8,0) {\(0\)};
    \node[below] at (2.6,0) {\(1\)};
\end{tikzpicture}
\end{center}
It follows that
\begin{equation}
\label{eqs:graded-pieces}
\begin{aligned}
    \mathrm{Gr}^W_{n+1} \mathbb{V}_{\infty} &= \mathrm{Ker}(\mathrm{H}^{n-1}(D_1\cap D_2)(-1) 
    \xrightarrow{\delta_{!}} \mathrm{H}^{n+1}(D_1)\oplus \mathrm{H}^{n+1}(D_2));\\
	\mathrm{Gr}^W_n \mathbb{V}_{\infty}   &= \frac{\mathrm{Ker}(\mathrm{H}^n(D_1)\oplus \mathrm{H}^n(D_2)
    \xrightarrow{\delta^*} \mathrm{H}^n(D_1\cap D_2))}{\mathrm{Im}(\mathrm{H}^{n-2}(D_1\cap D_2)(-1)
    \xrightarrow{\delta_{!}} \mathrm{H}^n(D_1)\oplus \mathrm{H}^n(D_2))};\\
    \mathrm{Gr}^W_{n-1} \mathbb{V}_{\infty} &= \mathrm{Coker}(\mathrm{H}^{n-1}(D_1)\oplus \mathrm{H}^{n-1}(D_2)
    \xrightarrow{\delta^*} \mathrm{H}^{n-1}(D_1\cap D_2)).
\end{aligned}
\end{equation}
The restriction and Gysin maps in the weight spectral sequence
\[
\delta^* = \sum_{j=1}^2 (-1)^{j-1}\delta^*_j,\quad \delta_{!}= \sum_{j=1}^2 (-1)^{j-1}{\delta_j}_!
\] 
are induced by the embeddings
\[
    \delta= \delta_1\sqcup \delta_2 \colon D_1\cap D_2\hookrightarrow D_2\coprod D_1.
\]

\textbf{The graded parts \(\mathrm{Gr}^W_{n+1}\) and \(\mathrm{Gr}^W_{n-1}\)}. 
By Proposition~\ref{prop:semi-red} and Lemma~\ref{lem:sm-quad-bdl}, 
the intersection \(D_1\cap D_2\) is a smooth quadric bundle 
\(g\colon E_0\to V\) of even relative dimension \(r=\frac{d+1}{2}\).
Deligne's decomposition formula~\eqref{eqs:deligne-decomp-thm} gives
\begin{equation}
\label{eqs:coh-smooth-quad}
	\mathrm{H}^{n-1}(E_0; \mathbb{Q}) \cong \bigoplus_{i=-r}^{r} \mathrm{H}^{\dim V-i}(V; R^{r+i}g_*\mathbb{Q}_{E_0})
\end{equation}
Since \(r\) is even, we have
\[
R^{r+i}g_*\mathbb{Q}_{E_0}=
\begin{cases}
    0,              & i~\textrm{odd} ;\\
    \mathbb{Q}_V,     & i ~\textrm{even}, i\neq 0;\\
    \mathbb{Q}_V^{2}, & i=0.
\end{cases}
\]
The hyperplane section \(V\) is projectively equivalent to a rational 
homogeneous space, see Subsection~\ref{subsec:severi-hyperplane}.
The Bruhat decomposition of a rational homogeneous space induces a cell
decomposition of \(V\) by even-dimensional real cells. Hence the
odd-degree cohomology of \(V\) vanishes. Since \(\dim V\) is odd,
every summand in~\eqref{eqs:coh-smooth-quad}
vanishes. Hence
\[
   \mathrm{Gr}^W_{n+1} \mathbb{V}_{\infty}= \mathrm{Gr}^W_{n-1} \mathbb{V}_{\infty}=0.
\]

\textbf{The central graded part \(\mathrm{Gr}^W_n\).}
Recall from 
Proposition~\ref{prop:semi-red} that \(D_1 =\widetilde{X}_0\) is the 
blow-up of \(X_0\) along \(V\) and \(D_2=E\) is the quadric fibration
over \(V\). The relevant Mayer--Vietoris sequence is
\begin{equation}
\label{eq:middle-row}
 \mathrm H^{n-2}(E_0)(-1)
 \xrightarrow{\ ({-\delta_2}_!, {\delta_1}_!)\ }
 \mathrm H^n(\widetilde{X}_0)\oplus\mathrm H^n(E)
 \xrightarrow{\ -\delta^*_2+\delta^*_1\ }
 \mathrm H^n(E_0).
\end{equation}
In the following, we denote by \(\eta=c_1(\mathcal O_E(1))=c_1(\mathcal O_E(E_0))\) 
the relative hyperplane class, and \(\eta_0=\delta_1^*(\eta)\).

The decompositions in Example~\ref{def-exam:dec-quad},
Theorem~\ref{thm:perv-decomp-quad-fib}, and
Proposition~\ref{prop:decomp-blow-up} give
\begin{align}
 \mathrm H^k(E_0)
 &=\bigoplus_{j=0}^{r}
     \mathrm H^{k-2j}(V)(-j)\oplus\ \mathrm H^{k-r}(V;\mathbb L),\quad k=n-2,n,
     \label{eq:g-cohom}\\
 \mathrm H^n(E)
 &=\bigoplus_{j=0}^{r+1}
     \mathrm H^{n-2j}(V)(-j)\oplus\ \mathrm H^{d-1}(B;L),
     \label{eq:f-cohom}\\
 \mathrm H^n(\widetilde{X}_0)
 &=\mathrm{IH}^n(X_0)\oplus\bigoplus_{j=1}^{r}\mathrm H^{n-2j}(V)(-j).
 \label{eq:blowup-cohom}
\end{align}

\emph{The quadric-fibration component.} The \(j\)-th hyperplane summand 
\(\mathrm{H}^{n-2j}(V)\) in \(\mathrm H^n(E)\) is realized by
\[
    x\mapsto f^*x\cup \eta^j.
\]
Because \(\delta_1^*(f^*x\cup\eta^j)=g^*x\cup\eta_0^j\),
the restriction map identifies \(\mathrm{H}^{n-2j}(V)\)
with the \(j\)-th hyperplane summand in \(\mathrm H^{n}(E_0)\) for 
\(0\le j\le r\). When \(j=r+1\), the class \(\eta^{r+1}_{0}\) has
the expansion
\begin{equation}
\label{eq:character-class}
\eta^{r+1}_{0}=\sum_{i=0}^{r} g^*c_{r+1-i}\cdot\eta^{i}_{0}+\vartheta, \quad
 c_{r+1-j}\in\mathrm H^{2r+2-2j}(V),\quad
 \vartheta\in\mathrm H^{r+2}(V;\mathbb L)
\end{equation}
Thus the restriction of \(\delta_1^*\) to the \(j=r+1\) summand
\(\mathrm{H}^{n-2r-2}(V)\) is
\begin{equation}
\label{eqs:top-degree-hyp}
\begin{aligned}
\mathrm{H}^{n-2r-2}(V) &\to \bigoplus_{j=0}^{r} \mathrm H^{n-2j}(V) \oplus \mathrm H^{n-r}(V, \mathbb L)\\
\alpha &\mapsto (c_{r+1}\cup \alpha,\ldots,c_1\cup \alpha, \vartheta\cup \alpha).
\end{aligned}
\end{equation}
The \((j-1)\)-th hyperplane summand \(\mathrm{H}^{n-2j}(V)\) in 
\(\mathrm H^{n-2}(E_0)\) is embedded by
\[
    x\mapsto g^*x\cup \eta^{j-1}.
\]
The self-intersection formula yields the relation 
\[
    {\delta_1}_*(g^*x\cup \eta^{j-1}_{0}) = f^*x\cup \eta^{j}, \quad 1\leq j\leq r+1.
\]
Thus \({\delta_1}_!\) identifies the \((j-1)\)-th hyperplane summand 
in \(\mathrm H^{n-2}(E_0)\) with the \(j\)-th hyperplane summand in 
\(\mathrm H^n(E)\) for \(1\le j\le r+1\). 

Let \(\zeta\in \mathrm H^{r+2}(E)\) be the normalized class in 
Proposition~\ref{prop:trace-orthognal}, which is orthogonal to the 
hyperplane classes and generates \(\mathrm H^0(B, L)\). Fix the split 
monomorphism \(i_*L[-r-2]\to Rf_*\mathbb Q_{E}\) corresponding to
\(\zeta\). We show that the map \(\delta_1^*\) on the primitive summand 
\(\mathrm H^{d-1}(B,L)\) has zero projection onto each hyperplane
summand in~\eqref{eq:g-cohom}. Write the hyperplane projection of
\(\delta_1^*(\zeta)\) as
\[
\sum_{j=0}^{\frac{r}{2}} g^*i_*b_{r-2j}\cdot \eta_0^j,
\]
where the coefficients \(b_{r-2j}\in \mathrm H^{r-2j}(B)\) are given
by the evaluation \((\zeta, \phi_j)\) with
\[
\phi_j\in \mathrm{Hom}(i_*L[-r-2], \mathbb{Q}_V[-2j])\cong  \mathrm{Hom}(L, i^!\mathbb{Q}_V[r+2-2j])=\mathrm H^{r-2j}(B, L^{\vee}).
\]
Recall that \({\delta_1}_*[E_0]=\eta\). The adjunction and trace 
orthogonality~\eqref{eqs:trace-orthogonal-decomp} imply
\[
 \mathrm{Tr}_g\left(\delta_1^*(\zeta)\cup\eta_{0}^k \right)
 =\mathrm{Tr}_f\left(\zeta\cup\eta^{k+1}\right)=0,
 \qquad 0\leq k\leq r,
\]
It implies that \(b_r=b_{r-2}=\cdots=b_0=0\).

The map to the primitive target is induced by the morphism 
\(i_*L[-r-2]\to \mathbb L[-r]\), which is identified with the 
ordinary Gysin map 
\[
    i_*i^!\mathbb L\to \mathbb L
\]
since \(i^!\mathbb L=i^*\mathbb L[-2]\) and \(i^*\mathbb L\cong L\) 
by Remark~\ref{rmk:pullback-prim-loc-sys}.  Hence
\begin{equation}
\label{eqs:ord-Gysin-map}
\delta^*_1=i_!\colon \mathrm H^{d-1}(B;L)(-1)\longrightarrow\mathrm H^{d+1}(V;\mathbb L)
\end{equation}
identifies to the ordinary cohomological Gysin map \(i_!\).


Let \(\xi\in \mathrm H^{r}(E_0)\) be the normalized class in 
Proposition~\ref{prop:trace-orthognal} that is orthogonal to hyperplane 
classes and generates \(\mathrm H^0(V, \mathbb L)\). Fix the split 
monomorphism \(\mathbb L[-r]\to Rg_*\mathbb Q_{E_0}\) given by \(\xi\).
For any \(z\in \mathrm{H}^{d-1}(V;\mathbb L)\), we claim that the 
hyperplane component of \({\delta_1}_*(z)\) can occur only in the \(j=0\) summand
\(\mathrm H^{n}(V)\) in \(\mathrm H^n(E)\). 
Write the hyperplane part of \({\delta_1}_*\xi\) as
\[
\sum_{j=0}^{\frac{r}{2}+1} f^*a_{r+2-2j}\cdot \eta^j, \qquad
a_{r+2-j}\in\mathrm H^{r+2-2j}(V).
\]
The trace orthogonality of \(\xi\) gives
\[
 \mathrm{Tr}_g( \xi\cup \eta_{0}^a)=
 \mathrm{Tr}_f( {\delta_1}_*\xi\cup \eta^a )=0,
   \quad 0\leq a\leq r,
\]
which implies that \(a_{r+1}=\cdots=a_0=0\). The remaining coefficient is
\[
 a_{r+2}=\frac12 g_*(\xi\cup \eta_0^{r+1})
    =\frac12 g_*(\xi\cup \vartheta)\in\mathrm H^{r+2}(V).
\]
Hence we have the morphism
\[
\begin{aligned}
\mathrm H^{d-1}(V;\mathbb L) &\to \mathrm H^{n}(V)\oplus \bigoplus_{j=1}^{r+1} \mathrm H^{n-2j}(V),\\
z&\mapsto (\frac12 z\cup g_*(\xi\cup \vartheta), 0,\ldots,0).
\end{aligned}
\]

The map to the primitive target is induced by the morphism
\(\mathbb L[-r-2]\to i_*L[-r-2]\), which is identified with the 
ordinary restriction map
\[
\mathbb L\to i_*i^*\mathbb L
\]
because \(i^*\mathbb L\cong L\). Hence
\begin{equation}
\label{eqs:ord-res-map}
i^*={\delta_1}_!\colon \mathrm H^{d-1}(V;\mathbb L) \longrightarrow\mathrm H^{d-1}(B;L)
\end{equation}
is the ordinary restriction map \(i^*\).

\emph{The blow-up component.} The \(j\)-th exceptional summand 
\(\mathrm H^{n-2j}(V)\) in \(\mathrm H^n(\widetilde{X}_0)\) is realized by
\[
x\mapsto {\delta_2}_*(g^*x\cup \eta_0^{j-1}),
 \quad 1\leq j\leq r.
\]
The self-intersection formula gives
\[
 \delta_2^*{\delta_2}_*(g^*x\cup \eta_0^{j-1})=-g^*x\cup \eta_0^{j},
 \quad 1\leq j\leq r.
\]
Thus the corresponding block of \(\delta_2^*\) is \(-I\) from the
\(j\)-th exceptional summand in \(\mathrm H^n(\widetilde{X}_0)\) to
the \(j\)-th hyperplane summand in \(\mathrm H^n(E_0)\).
On the intersection-cohomology summand \(\mathrm{IH}^n(X_0)\),
the map \(\delta_2^*\) has components
\[
(\rho, \gamma_1)\colon \mathrm{IH}^n(X_0)\to \mathrm H^n(V)\oplus \mathrm H^{d+1}(V, \mathbb L)
\]
induced by the morphism (see Corollary~\ref{cor:prop-base-change})
\[
\IC_{X_0}\to \iota_*\iota^* \IC_{X_0} =\mathbb Q_V[n]\oplus \mathbb L[n-r]
\]
which is the right adjoint of the identity map \( \iota^*\IC_{X_0}\to \iota^*\IC_{X_0}\) 
where \(\iota\colon V\hookrightarrow X_0\) is the embedding.

The Gysin map \({\delta_2}_!\) identifies the \((j-1)\)-th hyperplane summand 
\(\mathrm H^{n-2j}(V)\) in \(\mathrm H^{n-2}(E_0)\) with the \(j\)-th 
exceptional summand \(\mathrm H^{n-2j}(V)\) in 
\(\mathrm H^n(\widetilde{X}_0)\) for \(1\le j\le r\). For the class 
\({\delta_2}_*(g^*x\cup \eta_0^r)\), we have
\[
 {\delta_2}^*{\delta_2}_*(g^*x\cup \eta_0^r) = -g^*x\cup \eta_0^{r+1}=-\sum_{i=0}^{r} g^*(c_{r+1-j}\cdot x)\cup \eta^j-g^*x\cup \vartheta,
\]
where the classes \(c_{r+1-j}\) and \(\vartheta\) are the same 
as~\eqref{eq:character-class}. Therefore the map \({\delta_2}_!\) on 
the \(r\)-th hyperplane summand \(\mathrm H^{n-2r-2}(V)\) in
\(\mathrm H^{n-2}(E_0)\) has the following components:
\[
\begin{aligned}
{\delta_2}_!\colon \mathrm H^{n-2r-2}(V)&\to \mathrm{IH}^n(X_0)\oplus \bigoplus_{j=1}^{r} \mathrm H^{n-2j}(V),\\
    x&\mapsto (u(x), -c_r\cup x,\cdots, -c_1\cup x)
\end{aligned}
\]
such that \((\rho u(x), \gamma_1 u(x))=(-x\cup c_{r+1}, -x\cup \vartheta)\).
For the class \(\xi\), we have \({\delta_2}^*{\delta_2}_*\xi=-\xi\cup \eta_0=0\)
since \(\xi\) is orthogonal to hyperplane classes. Thus \({\delta_2}_!\) 
vanishes on \(\mathrm H^{d-1}(V, \mathbb L)\).

We now compute the middle cohomology of the three-term complex
in~\eqref{eq:middle-row}. Identity blocks allow us to cancel the
hyperplane and exceptional summands, leaving the primitive and
intersection-cohomology terms. Use the notation
\[
\begin{gathered}
 K_1=\mathrm{IH}^n(X_0),\qquad K_2=\mathrm H^{d-1}(B;L),\\
 P=\mathrm H^n(V),\quad
 A=\bigoplus_{j=1}^{r}\mathrm H^{n-2j}(V)(-j),\quad
 T=\mathrm H^{n-2r-2}(V)(-r-1),\\
 S=\mathrm H^{d-1}(V;\mathbb L)(-1),\quad
 R=\mathrm H^{d+1}(V;\mathbb L).
\end{gathered}
\]
The cohomology groups in~\eqref{eq:middle-row}
can be written as
\begin{equation}
\label{eqs:main-block-terms}
\begin{aligned}
 &\mathrm H^{n-2}(E_0)(-1)=S\oplus A\oplus T, \quad \mathrm H^n(E_0)=R\oplus P\oplus A,\\
 &\mathrm H^n(E)=K_2\oplus P\oplus A\oplus T,\quad \mathrm H^n(\widetilde{X}_0)=K_1\oplus A.\\
\end{aligned}
\end{equation}
The two copies of \(A\) in the middle term have different degrees.
The degrees of the hyperplane classes in the right-hand term shift by \(-1\).

With columns indexing source and rows indexing target, 
the differentials in~\eqref{eq:middle-row} have the
following block matrices. The blocks associated with \(D_1\) are
\begin{equation}
\label{eqs:main-D1-blocks}
 -{\delta_2}_!=
 \bordermatrix{
       &S        &A    &T   \cr
  K_1  &\alpha_1 &0    &u   \cr
  A    &0     &-I_A &\mu \cr
 },
 \qquad
 -\delta_2^*=
 \bordermatrix{
     &K_1      &A   \cr
  R  &\gamma_1 &0   \cr
  P  &\rho    &0   \cr
  A  & 0       &I_A \cr
 }.
\end{equation}
The morphisms \(\alpha_1,u,\mu,\gamma_1,\rho\) were described above.
The blocks associated with \(D_2\) are
\begin{equation}
\label{eqs:main-D2-blocks}
{\delta_1}_!=
 \bordermatrix{
      &S   &A   &T   \cr
  K_2 &i^* &0   &0   \cr
  P   &w   &0   &0   \cr
  A   &0   &I_A &0   \cr
  T   &0   &0   &I_T \cr
 },
\qquad
 \delta_1^*=
 \bordermatrix{
     &K_2      &P   &A   &T      \cr
  R  &i_! &0   &0   &\Theta \cr
  P  &0        &I_P &0   &C_0    \cr
  A  &0        &0   &I_A &C_A    \cr
 }.
\end{equation}
The map \((\Theta,  C_0, C_A)\colon T\to R\oplus P\oplus A\)
is given by~\eqref{eqs:top-degree-hyp}.
\[
 \Theta(t)=\vartheta\cup t,\quad
 C_0(t)=  c_{r+1}\cup t,\quad
 C_A(t)=(c_r\cup t,\ldots,c_1\cup t).
\]
The map \(w:S\to P\) is the cup product \(z\mapsto (\frac12 z\cup g_*(\xi\cup \vartheta))\)
computed above. The blocks \(i^*\), \(I_P\), and \(I_A\) are already
identified in the preceding paragraphs.


The identity \(\delta_1^*{\delta_1}_!+\delta_2^*{\delta_2}_!=0\), 
applied to the \(T\)- and \(S\)-columns, yields the relations
\begin{equation}
\label{eqs:main-block-relations}
\begin{aligned}
 \rho u+C_0&=0,              &\mu+C_A=0,         & \quad \rho \alpha_1+w=0,\\
 \gamma_1\alpha_1+i_!i^*&=0, &\gamma_1u+\Theta=0. &  \\
\end{aligned}
\end{equation}
Write an element of the middle term as
\[
 (k_1,a_1,k_2,p,a_2,t_0)
 \in(K_1\oplus A)\oplus(K_2\oplus P\oplus A\oplus T).
\]
The following invertible coordinate changes in the middle term
\[
 p'=p+\rho (k_1)+C_0(t_0),\qquad
 a_2'=a_2+a_1+C_A(t_0)
\]
turn the \(P\)- and \(A\)-components of \(d_+\) into \(p'\)
and \(a_2'\). By~\eqref{eqs:main-block-relations}, both coordinates
vanish on the image of \(d_-\). Thus
\(P\xrightarrow{I_P}P\) and \(A\xrightarrow{I_A}A\) split off.
The remaining differentials are
\[
\begin{aligned}
S\oplus A\oplus T &\to  K_1\oplus A\oplus K_2\oplus T \to R,\\
(s,x,t)&\mapsto
\bigl(\alpha_1(s)+u(t),\,-x+\mu(t),\,i^*(s),\,t\bigr),\\
(k_1,a_1,k_2,t_0)&\mapsto
 \gamma_1(k_1)+\gamma_2(k_2)+\Theta(t_0).
\end{aligned}
\]
Replacing the left-hand \(A\)-coordinate by
\(x'=x-\mu(t)\) splits off \(A\xrightarrow{-I_A}A\). Replacing \(k_1\) by
\(k_1'=k_1-u(t_0)\), and using
\(\gamma_1u+\Theta=0\), then splits off
\(T\xrightarrow{I_T}T\). In particular, all terms involving the
\(c_i\) and \(\vartheta\) are absorbed by these coordinate changes.
The remaining complex is
\[
 S\xrightarrow{\ (\alpha_1,i^*)\ }
 K_1\oplus K_2
 \xrightarrow{\ (\gamma_1,i_!)\ }R.
\]

All these coordinate changes are morphisms of Hodge structures.
Set \(m=r/2+1\). After restoring the Hodge twists, this complex is
\[
 \mathrm H^{d-1}(V)(-m)
 \xrightarrow{\,(-\iota_*,\,i^*)\,}
 \mathrm{IH}^n(X_0)\oplus\mathrm H^{d-1}(B)(-m)
 \xrightarrow{\,(-\iota^*,\, i_!)\,}
 \mathrm H^{d+1}(V)(-m+1).
\]
Here the constant local systems \(\mathbb Q_V\) and \(\mathbb Q_B\) 
replace \(\mathbb L\) and \(L\) respectively as \(V\) and \(B\) are 
simply connected.
By the self-intersection formula, \(i_!i^*\) is cup product with
\([B]\). Hard Lefschetz theorem makes it an isomorphism:
\[
 i_!i^*=[B]\cup-:
 \mathrm H^{d-1}(V)(-m)
 \xrightarrow{\ \sim\ }
 \mathrm H^{d+1}(V)(-m+1).
\]
The differential relation \(\iota^*\iota_*+i_!i^*=0\) and the preceding
isomorphism show that \(\iota_*\) and \(i^*\) are injective, while
\(\iota^*\) and \(i_!\) are surjective. Hence we obtain the decomposition
\[
\mathrm H^{d-1}(B;\mathbb Q)= \mathrm H^{d-1}(B;\mathbb Q)_{\mathrm{van}}
\oplus i^*\mathrm H^{d-1}(V;\mathbb Q),
\]
and
\[
\mathrm{IH}^n(X_0;\mathbb Q)=\ker(\iota^*)\oplus \iota_*\mathrm H^{d-1}(V;\mathbb Q)(-m).
\]
Set \(\mathrm{IH}^n(X_0;\mathbb Q)_v\coloneqq\ker\iota^*\). Therefore,
\[
\operatorname{Gr}^W_n \mathbb{V}_{\infty}
\cong
\mathrm H^{d-1}(B;\mathbb Q)_{\mathrm{van}}(-m)
\oplus \mathrm{IH}^n(X_0;\mathbb Q)_v.
\]

The intersection cohomology group \(\mathrm{IH}^n(X_0)\) carries a
pure Hodge structure of weight \(n\), and the inclusion arising from
the chosen splitting of the blow-up decomposition
\[
  \mathrm{IH}^n(X_0)\hookrightarrow \mathrm{H}^n(\widetilde{X}_0)
\]
is a morphism of weight \(n\) pure Hodge structures 
(see~\cite[Theorem 2.2.3]{dCAM05}). Recall that \(\widetilde{X}_0\) 
is a projective bundle over the Severi variety \(S\). Hence its
degree-\(n\) cohomology is a direct sum of Tate Hodge structures of
type \(\mathbb{Q}(-\frac{n}{2})\). Therefore \(\mathrm{IH}^n(X_0)_v\)
is isomorphic to \(\mathbb{Q}(-\frac{n}{2})^{\oplus k}\) where 
\(k=\dim \mathrm{IH}^n(X_0)_v\).
\end{proof}

\subsection*{Acknowledgments.} We thank Yilong Zhang for drawing our
attention to Collino's work and for sharing his preprint.
We also thank Zhiwei Zheng for helpful discussions.

\subsection*{AI disclosure.}
The authors developed the main ideas and proof methods, 
and completed the original draft without AI-generated
text. The ChatGPT-5.6-Sol were used to assist with
language polishing and a supplementary check of the
proof arguments. It identified gaps in the initial proof of
Theorem~\ref{thm:main-limit-mhs}, and suggested an argument 
to accomplish our approach. The conclusion of the main theorem remains unchanged.
The authors revised and verified the proof, 
and take full responsibility for correctness of this manuscript.

\end{document}